\documentclass[reqno,12 pt]{amsart}

\usepackage{amsfonts,amscd,mathrsfs}
\usepackage{amsthm}
\usepackage{amssymb}
\usepackage{lineno}[switch, displaymath, mathlines]
\usepackage{latexsym}
\usepackage{wasysym}
\usepackage{subcaption}
\usepackage{enumerate}
\usepackage{url}
\usepackage{cite}
\usepackage{amsmath}
\usepackage{verbatim}
\usepackage{graphicx}
\usepackage{epsf}
\usepackage{geometry}
\usepackage{color}

\usepackage{hyperref}
\definecolor{dark-blue}{rgb}{0,0,0.6}
\definecolor{Purple}{rgb}{0.2,0,0.25}
\hypersetup{breaklinks=true,
pagecolor=white,
colorlinks=true,
citecolor=dark-blue,
filecolor=black,%
linkcolor=Purple,%
urlcolor=blue
}

\theoremstyle{plain} 
\newtheorem{thm}{Theorem}[section]
\newtheorem{lem}[thm]{Lemma}
\newtheorem{defin}[thm]{Definition}
\newtheorem{cor}[thm]{Corollary}
\newtheorem{prop}[thm]{Proposition}

\theoremstyle{definition}
\newtheorem{remark}[thm]{Remark}
\newtheorem{example}[thm]{Example}

\newcommand{\wt}{\widetilde}

\newcommand{\Int}{\textnormal{int}}

\newcommand{\R}{\mathbb{R}}

\newcommand{\N}{\mathbb{N}}

\newcommand{\dist}{\textnormal{dist}}

\newcommand{\Rep}{\textnormal{Re}}

\newcommand{\wh}{\widehat}
\newcommand{\ol}{\overline}

\numberwithin{equation}{section}
\newcommand{\bref}[1]{\textbf{\ref{#1}}} 
\newcommand{\beqref}[1]{\textbf{(\ref{#1})}} 

\date{September 7, 2026} 



\begin{document}

\title[The BAP problem in normed spaces]{The best approximation pair  problem relative to two subsets in a normed space}

\author{Daniel Reem}
\address{The Center for Mathematics and Scientific Computation (CMSC), University of Haifa, Mt. Carmel, 3498838 Haifa, Israel.} 
\email{(Daniel Reem) dream@math.haifa.ac.il}
\author{Yair Censor}
\address{ Department of Mathematics, University of Haifa, Mt. Carmel, 3498838 Haifa, Israel.} 
\email{(Yair Censor) yair@math.haifa.ac.il}

\maketitle

\begin{abstract}
In the classical best approximation pair (BAP) problem, one is given two nonempty, closed, convex and disjoint subsets in a finite- or an infinite-dimensional Hilbert space, and the goal is to find a pair of points, each from each subset, which realizes the distance between the subsets. Motivated by our recent algorithm for solving the BAP problem [Censor, Mansour, Reem, J. Approx. Theory (2024)], we discuss the problem in more general normed spaces and with possibly non-convex subsets, and focus our attention on the fundamental issues of uniqueness and existence of the solution to the problem. We present several sufficient geometric conditions for the (at most) uniqueness of a BAP. These conditions are related to the structure and the relative orientation of the boundaries of the subsets and to the norm. We also present many sufficient  conditions for the existence of a BAP. In general, the paper re-examines several aspects related to the BAP problem, including the historical one, and shows, probably for the first time, how wide is the scope of the BAP problem in terms of the scientific communities which are involved in it (frequently independently) and in terms of its applications.
\end{abstract}

\section{Introduction}\label{sec:Intro}
\subsection{Background} 
The classical best approximation pair (BAP) problem is the following problem: there are two nonempty, disjoint, closed and convex subsets $A$ and $B$ in a finite- (i.e., Euclidean) or an infinite-dimensional real Hilbert space $(X,\|\cdot\|)$, and the goal is to find a pair of points, each from each subset, which realizes the distance between the subsets. In other words, the BAP problem is the following minimization problem: to find  a pair $(\wt{a},\wt{b})\in A\times B$ such that 
\begin{equation}
 \dist(A,B):=\inf\{\|a-b\|\,|\, a\in A, b\in B\}=\inf f(A\times B)=f(\wt{a},\wt{b})=\|\wt{a}-\wt{b}\|
\end{equation}
where $f:X^2\to [0,\infty)$ is defined by $f(x,y):=\|x-y\|$ for all $(x,y)\in X^2$. This problem has a long history which goes back to the classical 1959 work of Cheney and Goldstein \cite{CheneyGoldstein1959jour} (see also \cite{Goldstein1967book}) and, in some sense, even before (the paper \cite{Nicolescu1938jour} of Nicolescu from 1938), and continues with various other works such as, e.g., \cite{AharoniCensorJiang2018jour,BehlingBello-CruzSantos2021jour,BauschkeBorwein1993jour,
BauschkeBorwein1994jour,BauschkeBorweinLewis1997inproc,Cegielski2012book,
CegielskiSuchocka2008jour,GilbertJohnsonKeerthi1988jour,KopeckaReich2004jour,KopeckaReich2012jour,MatouskovaReich2003jour,
YoulaVelasco1986jour,BauschkeCombettesLuke2004jour,
GoldburgMarks1985jour,Luke2008jour,BauschkeSinghWang2022jour,
GubinPolyakRaik1967jour,CensorZaknoon2018jour,Dax2006jour,
Willner1968jour,LevitinPolyak1966jour}. See also \cite{CaseiroFacasVicenteVitoria2019jour,
FacasVicenteGoncalvesVitoria2014jour,GoncalvesFacasVicenteVitoria2015jour,DuPreKass1992jour,GrossTrenkler1996jour} for the Euclidean linear case, namely when both $A$ and $B$ are affine subspaces and the space is Euclidean, and \cite{Deutsch2001book,Garkavi1970jour,Singer1970book,Vlasov1973jour} for the case where one of the subsets is a point (this is the so-called ``best approximation problem'', that is, the problem of projecting a point onto another subset). 

The BAP problem has many applications in science and technology: for applications in signal processing see \cite{Combettes1994jour,CombettesBondon1999jour,GoldburgMarks1985jour,YoulaVelasco1986jour}, and for applications in solid modelling, computer graphics, robotics, collision detection, computer aided design, virtual reality and more, as well as for algorithms in convex and non-convex 2D and 3D Euclidean settings, see the following very partial list of references and the references therein: \cite{LinGottschalk1998inproc,SatoHirataMaruyamaArita1996inproc,
ElberGrandine2008inproc,JohnsonCohen1998inproc,ChenChenWangXuYongPaul2009jour,PatogluGillespie2002inproc,
Quinlan1994inproc,FanWangTongLiTang2024inproc,ChangChoiKimWang2011jour,
GilbertJohnsonKeerthi1988jour,ZeghloulRambeaud1996jour,
CameronCulley1986inproc,EhmannLin2000inproc,
LinManochaKim2018inbook,SonYoonKimElber2020jour,Schwartz1981jour,DobkinKirkpatrick1985jour}.

It can be seen that various scientific communities are involved in the research related to the BAP problem. One of the byproducts of this wide involvement, often with weak relations between the various participants and communities, is the widespread non-uniform terminology which characterizes this domain of research. For instance, a BAP may be called: ``a distance pair'', ``a proximal pair'', ``closest pair'', ``near point pair'', ``proximal points'', ``proximinal points'', and ``mutually nearest points''.

 Anyway, if one restricts oneself to the mathematical community in general, and to certain sub-communities in particular (such as the ones working in functional analysis, nonlinear analysis, optimization, convex analysis), then one can see that most of their attention regarding the BAP problem has been focused on the above-mentioned classical setting, but there are several works, which are scattered all over the literature,  which go beyond this setting, such as \cite{Kothe1969book,Luo2014jour,Narang1991jour,Sankar-RajAnthony-Eldred2014jour,Pai1974jour,Stiles1965b-jour,Xu1983jour,Xu1988jour} (normed spaces beyond Hilbert spaces, possibly nonconvex sets), \cite{BrezisMironescuShafrir2016jour,DamaBajracharya2018jour,LeviShafrir2014jour,
RubinsteinShafrir2007jour,Narang1976jour,Narang1983jour,Narang1984jour,Nicolescu1938jour} (metric and non-metric spaces) and \cite{Luke2008jour,MoralesSilvaGao2017inbook,VoiseiZalinescu2011jour,Zalinescu2022jour} (nonconvex sets in Euclidean spaces). The focus of many of these works is, however, on algorithmic or characterization aspects related to the BAP problem and not on the issues of existence and uniqueness of a BAP. 

We note that the BAP problem is related, sometimes closely, to other domains of research. The first deals with the  feasibility problem in which $A\cap B\neq\emptyset$ (in the special case where $A\cap B$ is convex this is the convex feasibility problem, or CFP for short), in contrast with the BAP problem which is mainly concerned with the case where $A\cap B=\emptyset$ (namely the inconsistent feasibility problem): see, for example, the following (very partial) list of references, as well as the references therein: \cite{BauschkeBorwein1996jour,
ButnariuCensorGurfilHadar2008jour,
ByrneCensor2001jour,
Cegielski2012book,CensorCegielski2015jour,
CensorZenios1997book,CensorReem2015jour,
Combettes1996jour(CFP),
GholamiTetruashviliStromCensor2013jour,GubinPolyakRaik1967jour}. 

Another domain of research deals with the so-called ``best proximity pair/points theorems''; here one starts with some space $X$, subsets $A$ and $B$ of $X$, a mapping $T$ defined on $X$ (possibly multivalued, possibly with a non-full domain of definition), and one wants to find conditions on $A$, $B$, $X$ and $T$ which ensure the existence of some $x\in A$ such that $\dist(x,Tx)=\dist(A,B)$, or variations of this equation (our existence results might enlarge the pool of such sufficient conditions): for a very partial list of related works, see \cite{DigarKosuru2020jour,KirkReichVeeramani2003jour,
SadiqBashaVeeramaniPai2001jour,SadiqBashaVeeramani200jour,RajuKosuruVeeramani2010jour,
Sankar-RajAnthony-Eldred2014jour,SharmaChandok2026jour,SuzukiKikkawaVetro2009jour,PatelPatel2023jour} and the references therein. 

The third domain of research deals with the issue  of genericity related to the BAP problem, say how much the set of all subsets $A$ and $B$ for which the BAP  problem has at least one  solution, a unique solution and so on is large in some well defined sense: see, for instance, \cite{De-BlasiMyjakPapini1992jour,Li2000jour,
LiXu2003jour,Peng2012jour,BoulosReich2015jour} and the references therein. 

The fourth domain of research is fixed point theory, since, as has already been observed by Cheney and Goldstein \cite[Theorem 2]{CheneyGoldstein1959jour} in the setting where $A$ and $B$ are nonempty, closed and convex subsets in a real Hilbert space (and where $P_A$ denotes the orthogonal projection operator onto $A$ and $P_B$ denotes the orthogonal projection operator onto $B$), if $(a,b)\in A\times B$ is a BAP relative to $(A,B)$, then $a=P_Ab$, $b=P_Ba$, $a$ is a fixed point of $P_AP_B$ and $b$ is a fixed point of $P_BP_A$;  conversely, if $a=P_Ab$ and $b=P_Ba$ (and then $a$ is a fixed point of $P_AP_B$ and $b$ is a fixed point of $P_BP_A$), then $(a,b)\in A\times B$ is a BAP relative to $(A,B)$. See \cite{BauschkeBorweinLewis1997inproc,CheneyGoldstein1959jour,
GubinPolyakRaik1967jour,Zaslavski2022jour,Zaslavski2024book}
 and the references therein for more details and for possibly more general results. 
 
\subsection{Contributions and motivation} In this work we discuss the BAP problem in normed spaces which are more general than Hilbert spaces and we allow possibly nonconvex subsets. We focus our attention on the fundamental issues of uniqueness and existence of the solution to the problem. The main contribution of our work is various new mathematical results, insights and many necessary and/or sufficient conditions - some of which are geometric in flavour - for either the (at most) uniqueness of a BAP or for its existence, as well as accompanied examples and counterexamples. 

Another contribution of our work is the semi-survey conducted in this section and in other parts (such as in Remark \bref{rem:Existence} below) which describes, probably for the first time, how wide is the scope of the BAP problem in terms of the scientific communities which are involved in it, frequently independently and without being aware of each other. Part of this semi-survey is a discussion on historical aspects related to the BAP problem which can be found in various remarks below. 

In a sense, our work re-examines the BAP problem from various aspects, including the historical aspect and certain mathematical aspects. However, up to very few places (such as the terminology expressed in Definition \bref{def:BestApproxPairIntervals}\beqref{item:BAP_Intervals}--\beqref{item:BAP_Intervals_Strictly_Nondegenerate}) we do not introduce and do not intend to introduce new mathematical techniques to this field of investigation; rather, we make use of either common techniques in the field (and related fields) or of techniques which easily follow from known ones, and use them for establishing our new results and insights. Actually, we wanted to show that such new results and insights can enrich this more than 65 years old and widely investigated domain of research and be of value despite the fact that they have been established using known techniques. Since our historical semi-survey shows that not everything has been smooth along the way in relation to various published results - including in seemingly simple scenarios (see, for instance, Remark \bref{rem:StilesPai}\beqref{item:Stiles} and Remark \bref{rem:Existence}\beqref{item:NoBAP}, \beqref{item:ConvexBG}, \beqref{item:LocallyCompact}, \beqref{item:Page345}) - it was very important for us to present detailed and rigorous proofs everywhere. 
   
Our motivation for investigating the existence and uniqueness issues came from the recent work \cite{CensorMansourReem2024jour} in which we discussed the alternating simultaneous Halpern-Lions-Wittman-Bauschke  (A-S-HLWB) algorithm for solving the BAP problem in the  Euclidean space $\R^k$ ($k\in\N$), under the additional assumption that both $A$ and $B$ are finite intersections of closed and convex subsets, that is, $A=\cap_{i=1}^m A_i$ and $B=\cap_{j=1}^{n} B_j$ for some $m,n\in\N$; this latter assumption leads to the computational advantage that one can orthogonally project iteratively onto the individual subsets $A_i$ and $B_j$, $i\in\{1,2,\ldots,m\}$, $j\in\{1,2,\ldots,n\}$ instead of projecting directly onto $A$ and $B$, a task which can be rather demanding from the computational point of view.

The practical importance of this scenario stems from its relevance to real-world situations, wherein the feasibility-seeking modelling is used and there are two disjoint constraint sets: one set, namely $A$,  represents constraints which must be satisfied (``hard'' constraints), while the other set (i.e., $B$) represents  constraints which, hopefully, will be satisfied (“soft” constraints). In this scenario the goal is to find a point which satisfies all the hard constraints and is located as close as possible to the intersection set $B$ of the soft constraints. This goal leads to the problem of finding a BAP relative to these two sets: again, see, e.g., \cite{GoldburgMarks1985jour,YoulaVelasco1986jour,Combettes1994jour,CombettesBondon1999jour} and the references therein for applications in signal processing.  

We showed in \cite[Theorem 32]{CensorMansourReem2024jour} that the A-S-HLWB algorithm converges to a BAP whenever it is known in advance that there is a unique BAP. This naturally leads to the task of providing conditions, hopefully easy-to-verify, which ensure that there is a unique BAP. In \cite[Proposition 16(iii)]{CensorMansourReem2024jour} we presented a sufficient condition for the uniqueness of the BAP: that $A$ and $B$ are compact,  strictly convex (i.e., that their boundaries do not contain nondegenerate line segments) and satisfy $\dist(A,B)>0$. While the strict convexity  condition covers a large class of cases, there are many cases in which there is a unique BAP but the above-mentioned condition does not hold, and a simple example was given in \cite[Figure 3.1]{CensorMansourReem2024jour}. 

Here we generalize \cite[Proposition 16(iii)]{CensorMansourReem2024jour} to all normed spaces, and present various other sufficient (and sometimes necessary) geometric conditions for the (at most) uniqueness of a BAP in a wide class of normed spaces. These conditions are related to the structure of the boundaries of the subsets, their relative orientation, and to the structure of the unit sphere of $X$. Roughly speaking, one of these conditions (Corollary \bref{cor:StrictlyConvexSets} below) says that if the unit sphere of $X$ does not contain nondegenerate  intervals (that is, if $(X,\|\cdot\|)$ is strictly convex), and if either $A$ or $B$ is strictly convex, or if the boundaries of both of them contain nondegenerate intervals but no interval from the boundary of one subset is parallel to an interval contained in the boundary of the other subset, then there is at most one BAP. Our analysis, which is illustrated by various examples and counterexamples, also covers the case where both $A$ and $B$ are finite intersections of closed and convex subsets. As can be seen from this discussion, our results enrich the pool of available sufficient conditions for the uniqueness of a BAP which can be used in association with the A-S-HLWB algorithm.

More details about known results related to the uniqueness issue of a BAP can be found in Remark \bref{rem:StilesPai}\beqref{item:Stiles}, Remark \bref{rem:StrictConvexNormedSpace}\beqref{item:A-A_B-B}, and Remark \bref{rem:StilesPai}\beqref{item:PaiUniqueness} below, as well as in the discussion before Proposition  \bref{prop:UniquenessExistence} below. See also  \cite[Proof of Theorem 1]{AharoniCensorJiang2018jour} and \cite[Proposition 16 and Theorem 32]{CensorMansourReem2024jour} (where in both cases the setting is a Euclidean space).

The issue of existence is considered in Theorem \bref{thm:Existence} below, which presents 32 sufficient conditions for the existence of a BAP. These conditions  involve, among others, reflexive Banach spaces, dual spaces, finite-(co)dimensional spaces, convex and non-convex geometric objects, and affine subspaces. Since a BAP is not guaranteed to exist in general, as several examples presented in Remark \bref{rem:Existence}\beqref{item:NoBAP}--\beqref{item:Nonreflexive} show (see additional items in Remark \bref{rem:Existence} for various extensions and historical aspects related to the existence issue), one should not be surprised that some conditions, of topological, analytical and geometric nature, should be imposed in order to ensure existence. We note that about third of the sufficient conditions mentioned in Theorem \bref{thm:Existence} are known and we mentioned them either for the sake of completeness or in order to show how they follow from our new results. Theorem \bref{thm:Existence} covers and extends a large part of the existence results that we have seen in the literature. 

We note that the existence of a BAP is important also because without it various algorithms aimed at solving the BAP problem, such as the Dykstra algorithm \cite{BauschkeBorwein1994jour}, the alternating  projection algorithm \cite[Theorem 4]{CheneyGoldstein1959jour},\cite[Theorems 1.2--1.4]{Stiles1965b-jour} (inspired by von Neumann \cite[Lemma 22, p. 475]{vonNeumann1949jour}, \cite[Theorem 13.7, pp. 55--56]{vonNeumann1950book}), and the A-S-HLWB algorithm \cite[Theorem 32]{CensorMansourReem2024jour}, may not converge.

\subsection{Paper layout}
In Section \bref{sec:Preliminaries} we present our notation and recall a few known concepts. Section \bref{sec:Auxiliary} presents various auxiliary results (whose proofs appear in the Appendix at the end of the paper). Section \bref{sec:Uniqueness} presents several conditions which ensure the (at most) uniqueness of a BAP, and Section \bref{sec:Existence} presents many sufficient conditions for the existence of a BAP. Some of the results are illustrated by various examples, counterexamples and figures  presented in Section \bref{sec:Examples}. 

\section{Preliminaries}\label{sec:Preliminaries}
In this section we present some terminology and recall several known concepts and results. Unless otherwise stated, our setting is a normed space $(X,\|\cdot\|)$, $X\neq \{0\}$, but since some of the notions below hold in a more general setting, such as metric spaces and vector spaces, we sometimes consider these settings too. 

We denote by $X^*$ the dual of $X$. Given a subset $A\subseteq X$, we denote by $\overline{A}$ its closure, by $\partial A$ its boundary, and by $\Int(A)$ its interior. Given another subset $B$ of $X$, the distance between $A$ and $B$ is defined by $\dist(A,B):=\inf\{\|a-b\|\,|\, a\in A, b\in B\}$ (if either $A$ or $B$ is empty, then $\dist(A,B):=\infty$). We say that $B$ is proximinal with respect to $A$ if for every $a\in A$ there exists $b\in B$ such that $d(a,B):=d(\{a\},B)=\|a-b\|$. We denote $A+B:=\{a+b\,|\,a\in A, b\in B\}$ and $A-B:=\{a-b\,|\,a\in A, b\in B\}$. If $A$ is a linear subspace, then we say that $A$ is topologically complemented if $A$ is closed and there exists a closed linear subspace $F$ such that $A\oplus F=X$, that is, $A\cap F=\{0\}$ and $A+F=X$. In this case we denote by $\Pi_A:X\to A$ the linear projection onto $A$ along $F$, that is, if $z\in X$ is (uniquely) represented as $z=z_1+z_2$ for some $z_1\in A$ and $z_2\in F$, then $\Pi_A(z)=z_1$. Similarly, $\Pi_F:X\to F$ denotes the linear projection onto $F$ along $A$. If $F$ is finite dimensional, then we say that $A$ has a finite codimension. 

We say that $A$ is an affine subspace if $A=u+\wt{A}$ for some $u\in X$ and a linear subspace $\wt{A}$; in this case we say that $\wt{A}$ is  the linear part of $A$, and the dimension/codimension of $A$ is defined to be the dimension/codimension of $\wt{A}$. We say that $A$ is polyhedral, or a polyhedron (not necessarily bounded), if it is the intersection of finitely many closed real halfspaces, namely there are $m\in\N$, $f_i\in X^*$ and $\beta_i\in\R$ for all $i\in \{1,\ldots,m\}$ such that  $A=\cap_{i=1}^m\{x\in X\,|\, \Rep f_i(x)\leq \beta_i\}$.

Given two points $a_0$ and $a_1$ in $X$, the interval, or line segment, connecting them is denoted by $[a_0,a_1]:=\{a(t):=a_0+t(a_1-a_0)\,|\, t\in [0,1]\}$. This interval is said to be nondegenerate if $a_0\neq a_1$. We denote by $(a_0,a_1):=\{a(t)\,|\, t\in (0,1)\}$ the open interval connecting $a_0$ and $a_1$, and by $[a_0,a_1):=\{a(t) \,|\,  t\in [0,1)\}$ and $(a_0,a_1]:=\{a(t) \,|\, t\in (0,1]\}$ the respective half-open intervals. We say that $A\subseteq X$ is strictly convex if for all distinct points $a_0,a_1\in A$, the open interval $(a_0,a_1)$ is contained in $\Int(A)$. Any strictly convex set is obviously convex. A real line in $X$ is an affine set whose linear part is a one-dimensional real linear subspace. Given two real lines $L$ and $M$ in $X$, we say that they are strictly parallel if they are disjoint and there is a real two-dimensional affine subspace in which they are located; we say that $L$ and $M$ are parallel if either they coincide or they are strictly parallel. Given two nondegenerate intervals $[a_0,a_1]$ and $[b_0,b_1]$, we say that they are parallel (respectively, strictly parallel) if there are two parallel (resp., strictly parallel) real lines $L$ and $M$, such that $[a_0,a_1]\subset L$ and $[b_0,b_1]\subset M$. 
 
We say that $A$ is a generalized hypercylinder with a base $C$ and an axis $L$ if $A=L+(C-c)$, where $C$ is a subset of $X$ and $L$ is a real line such that $L\cap C=\{c\}$ for some $c\in X$. In the special case where $X$ is a real Hilbert space, $L$ is a real line passing through some point $c\in X$, and $C$ is a ball with respect to $L^{\perp}$ whose radius is $r>0$ and its center is $c$, then $A$ is a right (standard) hypercylinder; if $L$ is a real line passing through some $c\in X$ and if $C$ is a ball with respect to some hyperplane $H$ whose radius is $r>0$ and its center is $c$ and the angle between $L$ and $C$ is not $\pi/2$ but rather is strictly between 0 and $\pi/2$, then $A:=L+(C-c)$ is a tilted hypercylinder. See Figures \bref{fig:ParallelStraightCylinders}--\bref{fig:GeneralizedCylindersNonParallel} for a few examples in $\R^3$ (Euclidean and non-Euclidean norms).

We say that $((a_k,b_k))_{k\in\N}$ is a distance minimizing sequence in $A\times B$ if  
\begin{equation}\label{x_ky_xdist(A,B)}
\lim_{k\to\infty}\|a_k-b_k\|=\dist(A,B).
\end{equation}
The definition of $\dist(A,B)$ obviously implies the  existence of at least one (not necessarily convergent) distance minimizing sequence when $A\neq\emptyset$ and $B\neq\emptyset$. We say that $(A,B)$ satisfies the distance coercivity condition if $A\cup B$ is unbounded and
\begin{equation}\label{eq:coercive}
\lim_{\|(x,y)\|\to\infty, (x,y)\in A\times B}\|x-y\|=\infty, 
\end{equation}
where $\|(x,y)\|:=\sqrt{\|x\|^2+\|y\|^2}$ for all $(x,y)\in X^2$. 

If a sequence $(x_k)_{k\in \N}$ converges weakly to $x\in X$, then we write $x=(w)\lim_{k\to\infty}x_k$. If $X$ is a dual of a Banach space $Y$ and $(x_k)_{k\in \N}$ converges to $x\in X$ in the weak-star topology, then we write $x=(w^*)\lim_{k\to\infty}x_k$. We say that $A$ is weakly sequentially closed if for every $x\in X$ and  $(x_k)_{k\in \N}$ in $A$, the condition $x=(w)\lim_{k\to\infty}x_k$ implies that $x\in A$. Note that if $A$ is weakly closed, then it is weakly sequentially closed (since in general, if a subset is closed with respect to some topology, then it is sequentially closed with respect to that topology), but the converse is not necessarily true even in Hilbert spaces \cite[Example 3.33, p. 60]{BauschkeCombettes2017book}. We say that $A$ is weakly sequentially compact if every sequence in $A$ has a subsequence which converges weakly to some $z\in A$ (this is a standard notion but occasionally, as in  \cite[Definition II.3.25, pp. 67--68]{DunfordSchwartz1958book}, one requires less and gets less, namely that the limit $z$ exists in $X$ and not necessarily in $A$). We say that $A$ is locally compact if for every $z\in A$ there is a (non-degenerate) closed ball $D$ centered at $z$ such that $D\cap A$ is compact. We say that $A$ is boundedly compact if $D\cap A$ is compact for every closed ball $D$.

Finally, we say that the underlying space $X$ is strictly convex if its unit ball (and hence any other ball) is a strictly convex subset. Equivalently, the boundary of the ball does not contain nondegenerate intervals.  Well-known examples of strictly convex spaces are Euclidean spaces, Hilbert spaces, the sequence spaces $\ell_p$ (sequences with possibly finitely many components) where $p\in (1,\infty)$, the function spaces $L_p(\Omega)$ where $p\in (1,\infty)$ and $\Omega$ is a Lebesgue measurable set in $\R^k$ for some $k\in\N$, uniformly convex spaces, and sums of strictly convex spaces with the $\|\cdot\|_p$ norm ($p\in (1,\infty)$). Well-known examples of spaces which are not strictly convex are the $\ell_1$, $\ell_{\infty}$,  $L_1(\Omega)$ and $L_{\infty}(\Omega)$ spaces. For more details and examples, see, for instance,  \cite{Beauzamy1982book,DunfordSchwartz1958book,Prus2001incol,GoebelReich1984book,LindenstraussTzafriri1979book,Clarkson1936jour}.

\section{Auxiliary results}\label{sec:Auxiliary}
In this section we formulate several auxiliary results which we need in order to prove our main results. The proofs of these results appear in the Appendix at the end of the paper. 

We start with the following two simple (and probably known) lemmata. 
 
\begin{lem}\label{lem:DistClosures}
Given two nonempty subsets $A$ and $B$ in a metric space $(X,d)$, one has $\dist(A,B)=\dist(\ol{A},\ol{B})$, where $\dist(A,B):=\inf\{d(a,b)\,|\, (a,b)\in A\times B\}$.  Moreover, if $A\cap B=\emptyset$ and there exists a BAP with respect to $(A,B)$, namely a pair $(a_0,b_0)\in A\times B$ such that $d(a_0,b_0)=\dist(A,B)$, then $\dist(A,B)>0$.
\end{lem}

\begin{lem}\label{lem:DistBoundary}
Suppose that $(X,\|\cdot\|)$ is a normed space and that $A$ and $B$ are nonempty and disjoint subsets of $X$. Then $\dist(A,B)=\dist(\partial A,\partial B)$. 
\end{lem}

\begin{remark}\label{rem:DistBoundaryMetric}
Lemma \bref{lem:DistBoundary} does not hold in every path connected metric space. Indeed, suppose that $X$ is the subset of $\R^2$ defined by $([-2,2]\times\{-2\})\cup (\{2\}\times [-2,2])\cup ([-2,2]\times \{2\})\cup (\{-2\}\times [1,2])\cup(\{-2\}\times[-2,-1])$, that is, $X$ is the subset obtained by removing the line segment $\{-2\}\times (-1,1)$ from the boundary (in $\R^2$) of the square $[-2,2]^2$. Let $d:X^2\to [0,\infty)$ be the metric induced by the Euclidean norm, namely $d((x_1,x_2),(y_1,y_2)):=\sqrt{(x_1-y_1)^2+(x_2-y_2)^2}$ for all $(x_1,x_2),(y_1,y_2)\in X$. Let $A:=\{-2\}\times [-2,-1]$ and $B:=\{-2\}\times [1,2]$. Then $\partial A=\{(-2,-2)\}$, $\partial B=\{(-2,2)\}$, $A\cap B=\emptyset$ and  $2=\dist(A,B)<\dist(\partial A,\partial B)=4$. 
\end{remark}

We continue with the following definition. See Section \bref{sec:Examples} below for illustrations. 

\begin{defin}\label{def:BestApproxPairIntervals}
Let $(X,\|\cdot\|)$  be a normed space, let $a_0, a_1, b_0, b_1\in X$ and let $A$ and $B$ be two nonempty subsets of $X$. We say that:
\begin{enumerate}[(i)]
\item $(a_0,b_0)$ is a best approximation pair (BAP) with respect to (or relative to) $(A,B)$ if $a_0\in A$, $b_0\in B$ and $\|a_0-b_0\|=\dist(A,B)$. 
\item \label{item:BAP_Intervals} $([a_0,a_1],[b_0,b_1])$ is a BAP of intervals with respect to $(A,B)$ if for all $t\in [0,1]$ one has $a(t)\in A$, $b(t)\in B$ and $\|a(t)-b(t)\|=\dist(A,B)$, where $a(t):=a_0+t(a_1-a_0)$ and $b(t):=b_0+t(b_1-b_0)$. 

\item  $([a_0,a_1],[b_0,b_1])$ is a nondegenerate BAP of intervals with respect to $(A,B)$ if it is a BAP of intervals with respect to  $(A,B)$ and either $a_0\neq a_1$ or $b_0\neq b_1$. 

\item\label{item:BAP_Intervals_Strictly_Nondegenerate} $([a_0,a_1],[b_0,b_1])$ is a strictly nondegenerate BAP of intervals with respect to  $(A,B)$ if it is a BAP of intervals relative to $(A,B)$ and both $a_0\neq a_1$ and  $b_0\neq b_1$. 
\end{enumerate}
\end{defin}

\begin{lem}\label{lem:a(t)b(t)} 
Suppose that $(X,\|\cdot\|)$ is a normed space and that $A$ and $B$ are nonempty and convex subsets of $X$ such that  $a_0, a_1\in A$ and $b_0, b_1\in B$.  Then the following two conditions are equivalent:
\begin{enumerate}[(i)]
\item\label{item:0,1} $(a_0,b_0)$ and $(a_1,b_1)$ are BAPs with respect to $(A,B)$. 

\item\label{item:[0,1]} $(a(t),b(t))$ is a BAP with respect to $(A,B)$ for all $t\in [0,1]$, where $a(t):=a_0+t(a_1-a_0)$ and $b(t):=b_0+t(b_1-b_0)$. 
\end{enumerate}
\end{lem}

The next lemma is related to, but definitely different from, \cite[Proposition 3.1]{SadiqBashaVeeramani200jour} (which, by the way, has a minor gap in its proof, where it is claimed without a proof that the line segment $K$ defined there must intersect $\partial A$; in this connection, see the proof of Lemma \bref{lem:DistBoundary}).

\begin{lem}\label{lem:a_0b_0Boundary}
Suppose that $(X,\|\cdot\|)$ is a normed space and that $A$ and $B$ are nonempty and disjoint subsets of $X$. If $(a_0,b_0)$ is a BAP with respect to $(A,B)$, then it is a BAP with respect to $(\partial A,\partial B)$; in particular,  $a_0\in \partial A$ and $b_0\in \partial B$. Conversely, if $(a_0,b_0)$ is a BAP with respect to $(\partial A,\partial B)$ and $A$ and $B$ are also closed (in addition to being  nonempty and disjoint), then $(a_0,b_0)$ is a BAP with respect to $(A,B)$. 
\end{lem}

\begin{lem}\label{lem:BestApproxIntervalBoundary} 
Let $A$ and $B$ be nonempty, convex and disjoint subsets of a normed space $(X,\|\cdot\|)$. Assume  that $a_0, a_1\in A$ and $b_0, b_1\in B$. If $(a_0,b_0)$ and $(a_1,b_1)$ are BAPs with respect to $(A,B)$, then $([a_0,a_1],[b_0,b_1])$ is a BAP of intervals with respect to both $(A,B)$ and $(\partial A,\partial B)$. In particular, $[a_0,a_1]\subseteq \partial A$ and $[b_0,b_1]\subseteq \partial B$.
\end{lem}

\begin{lem}\label{lem:StrictlyNondegenerate}
Let $A$ and $B$ be nonempty, convex and disjoint subsets of a strictly convex normed space $(X,\|\cdot\|)$. If $([a_0,a_1],[b_0,b_1])$ is a nondegenerate BAP of intervals with respect to $(A,B)$, then it is a strictly nondegenerate  BAP of intervals relative to both $(A,B)$ and  $(\partial A, \partial B)$. In particular, $[a_0,a_1]$ is a nondegenerate interval contained in $\partial A$, and $[b_0,b_1]$ is a nondegenerate interval contained in $\partial B$.
\end{lem}

\begin{lem}\label{lem:parallel}
Suppose that $(X,\|\cdot\|)$ is a strictly convex normed space and that $A$ and $B$ are nonempty, convex and disjoint  subsets of $X$. If $(a_0,b_0)$ and $(a_1,b_1)$ are two distinct BAPs with  respect to $(A,B)$, then both $[a_0,a_1]$ and $[b_0,b_1]$ are nondegenerate intervals, they are strictly parallel, $[a_0,a_1]\subseteq \partial A$ and $[b_0,b_1]\subseteq \partial B$. 
\end{lem}

The following lemma, which actually holds in any topological vector space with essentially the same proof, might be known.

\begin{lem}\label{lem:ThreePointsIntervalBoundary}
Let $C$ be a nonempty and convex subset of a normed space $(X,\|\cdot\|)$. If $x$,$y$ and $z$ are three distinct points in $\partial C$ satisfying $y\in [x,z]$, then $[x,z]\subseteq\partial C$.
\end{lem}

The final two auxiliary assertions will be used in Section \bref{sec:Existence}. 

\begin{lem}\label{lem:(Un)BoundedMinimizing}
Let $(X,\|\cdot\|)$ be a normed space, $A$ and $B$ be nonempty subsets of $X$, and $((a_k,b_k))_{k\in\N}$ be a distance minimizing sequence in $A\times B$. Then:
\begin{enumerate}[(i)]
\item\label{item:BoundedOrUnbounded} Either both $(a_k)_{k\in\infty}$ and $(b_k)_{k\in\N}$ are bounded or both of them are unbounded. 
\item\label{item:BoundedSequences} If either $A\cup B$ is bounded, or  $A\cup B$ is unbounded and \beqref{eq:coercive} holds, then both $(a_k)_{k\in\infty}$ and $(b_k)_{k\in\N}$ are bounded. 
\item\label{item:WeaklyLSC} If $((a_k,b_k))_{k\in\N}$ has a subsequence which converges weakly in $X^2$ to some point $(a,b)\in A\times B$, then $(a,b)$ is a BAP relative to $(A,B)$.
\end{enumerate}
\end{lem}

We need the following definition in order to establish Lemma \bref{lem:WeaklyBolzanoWeierstrass} below it. 
\begin{defin}\label{def:NormedLocalWeakSequentialCompact}
We say that a nonempty subset $C$ of a normed space $(X,\|\cdot\|)$  is norm-locally weakly sequential compact if for every $x\in C$ there is a closed ball $D\subseteq X$ centered at $x$  such that $D\cap C$ is weakly sequentially compact. 
\end{defin}
\begin{remark}\label{rem:LocallyNormWeaklySequentialCompact}
\begin{enumerate}[(i)]
\item Here are a few examples of norm-locally weakly sequential compact subsets: any locally compact set in any normed space (because strong convergence implies weak convergence); any weakly sequentially closed set in a reflexive Banach space (because closed balls are weakly compact and hence weakly sequentially compact); any closed and convex set in a reflexive Banach space (because closed and convex sets in a reflexive Banach space are weakly sequentially closed); a not necessarily countable union $\cup_{i\in I} C_i$ of nonempty and closed and convex sets $C_i$ in a reflexive Banach space having the property that $0<\inf\{\dist(C_i,C_j)\,|\,j\in I\backslash\{i\}\}$ for all $i\in I$ (because in this case any $x$ in the union satisfies $x\in C_i$ for some $i\in I$, and then the closed ball $D$ of radius $r_i:=0.5\inf\{\dist(C_i,C_j)\,|\,j\in I\backslash\{i\}\}$ centered at $x$ satisfies $D\cap (\cup_{j\in I} C_j)=D\cap C_i$, which is a nonempty, closed, convex and bounded set in a reflexive Banach space and hence weakly sequentially compact). 
\item When $X$ is a Banach space, then Definition \bref{def:NormedLocalWeakSequentialCompact} coincides with a different, but closely related, definition: the definition of locally weakly compact sets. More precisely, we say that a nonempty subset $C$ of a normed space is locally weakly compact if for every $x\in C$ there is a closed ball $D$ centered at $x$ such that $D\cap C$ is weakly compact (this definition is not very well known, and it appears here and there in the literature, such as in \cite[p. 1302]{LuoLiXuYao2011jour} and \cite{Bruck1973jour}). Since $X$ is a Banach space,  the well known Eberlein–{\v{S}}mulian theorem \cite[Theorem V.6.1, p. 430]{DunfordSchwartz1958book} ensures that a nonempty subset $C$ of $X$ is weakly compact if and only if it is weakly sequentially compact, and we have the equivalence between the definitions. As it turns out, the Eberlein–{\v{S}}mulian theorem holds also in real normed spaces: see \cite[pp. 145--149]{Holmes1975book} (based on modification of the real Banach space proof given in \cite{Whitley1967jour}) or \cite{Kremp1986jour}; hence, as above, also in this setting both Definition \bref{def:NormedLocalWeakSequentialCompact} and the definition of locally weakly compact sets coincide. 
\end{enumerate}
\end{remark}

\begin{lem}\label{lem:WeaklyBolzanoWeierstrass}
Suppose that  $C$ is a nonempty, convex, closed and norm-locally weakly sequentially compact subset of a normed space $(X,\|\cdot\|)$. Then any bounded sequence in $C$ has a weakly convergent subsequence whose weak limit is in $C$.
\end{lem}

\begin{remark}\label{rem:WeaklyBolzanoWeierstrass}
A result related to (but definitely different from)  Lemma \bref{lem:WeaklyBolzanoWeierstrass} says that if $(X,d)$ is a locally compact and almost complete geodesic metric space, then every infinite set in $X$ has an accumulation point with respect to the topology induced by the geodesic metric: see \cite[Theorem 4.3]{Myers1945jour}.   
\end{remark}
                 
\section{Uniqueness of the best approximation pair}\label{sec:Uniqueness}
This section presents our results regarding the (at most) uniqueness of the BAP. 

\begin{prop}\label{prop:UniquenessInNormedSpace}
Suppose that $A$ and $B$ are two nonempty, convex and disjoint subsets of a normed space $(X,\|\cdot\|)$. If there does not exist a nondegenerate BAP of intervals with respect to $(\partial A,\partial B)$, then there exists at most one BAP relative to $(A,B)$. Conversely, if $A$ and $B$ are also closed and there exists at most one BAP with respect to $(A,B)$, then there does not exist a nondegenerate BAP of intervals relative to $(\partial A,\partial B)$.
\end{prop}
\begin{proof}
Assume first that there does not exist a nondegenerate BAP of intervals with respect to $(\partial A,\partial B)$. If $\dist(A,B)$ is not attained, then there does not exist any BAP relative to $(A,B)$, and hence obviously there exists at most one BAP relative to $(A,B)$. Otherwise $\dist(A,B)$ is attained, and hence there exists at least one BAP $(a_0,b_0)$ with respect to $(A,B)$. If, for the sake of contradiction, there exists another (different) BAP $(a_1,b_1)$ with respect to $(A,B)$, then Lemma \bref{lem:BestApproxIntervalBoundary}  implies that $([a_0,a_1],[b_0,b_1])$ is a BAP of intervals with respect to $(\partial A,\partial B)$, and $([a_0,a_1],[b_0,b_1])$ is nondegenerate since either $a_0\neq a_1$ or $b_0\neq b_1$. This contradicts the assumption that there does not exist a nondegenerate BAP of intervals with respect to $(\partial A,\partial B)$. Hence $(a_0,b_0)$ is the unique BAP relative to $(A,B)$. Conversely, suppose that $A$ and $B$ are also closed and that  there exists at most one BAP with respect to $(A,B)$. If, for the sake of contradiction, there exists  a nondegenerate BAP of intervals $([a_0,a_1],[b_0,b_1])$ with respect to $(\partial A,\partial B)$, then either $a_0\neq a_1$ or $b_0\neq b_1$, and in both cases $(a_0,b_0)$ and $(a_1,b_1)$ are two distinct BAPs with respect to $(\partial A,\partial B)$ and hence, according to Lemma \bref{lem:a_0b_0Boundary}, also with respect to $(A,B)$. This is a contradiction to the assumption that there exists at most one BAP with respect to $(A,B)$. 
\end{proof}

\begin{thm}\label{thm:IntersectionStrictlyConvexSets}
Let $(X,\|\cdot\|)$ be a normed space. Suppose that $m$ and $n$ are natural numbers and that $A_1,A_2,\ldots,A_m$ and $B_1,B_2,\ldots,B_n$ are nonempty and strictly convex subsets of $X$. If $A:=\cap_{i=1}^m A_i$ and $B:=\cap_{j=1}^n B_j$ are nonempty and disjoint, then there exists at most one BAP relative to $(A,B)$. If, in addition, $\dist(A,B)$ is attained, then there exists exactly one BAP relative to $(A,B)$.
\end{thm}
\begin{proof}
Since, as follows immediately from the definition, any finite intersection $\cap_{k=1}^{\ell}C_k$ of strictly convex sets is strictly convex (indeed, if $x,y\in \cap_{k=1}^{\ell}C_k$, then $x,y\in C_k$ for all $k\in\{1,\ldots,\ell\}$, and hence the strict convexity of $C_k$ implies that the open interval $(x,y)$ is contained in $\Int(C_k)$ for all $k\in\{1,\ldots,\ell\}$, and hence $(x,y)\subseteq \cap_{k=1}^{\ell}\Int(C_k)\subseteq \Int(\cap_{k=1}^{\ell}C_k)$, as required), it follows that $A$ and $B$ are strictly convex and, in particular, convex. Now, if $\dist(A,B)$ is not attained, then obviously  there exists at most one (actually zero) BAP relative to $(A,B)$. Otherwise, there exists at least one BAP $(a_0,b_0)$ relative to $(A,B)$. Assume for the sake of contradiction that there exists another  BAP $(a_1,b_1)\neq(a_0,b_0)$ relative to $(A,B)$. Then either $[a_0,a_1]$ or $[b_0,b_1]$ is nondegenerate. Since the conditions of Lemma \bref{lem:BestApproxIntervalBoundary} hold, we conclude from it that $[a_0,a_1]\subseteq\partial A$ and $[b_0,b_1]\subseteq\partial B$. This  contradicts the fact that both $A$ and $B$ are strictly convex and hence their boundaries do not contain any nondegenerate intervals. Therefore there exists exactly one BAP relative to $(A,B)$, as claimed. 
\end{proof}

\begin{cor}\label{cor:StrictlyConvexSets}
Let $(X,\|\cdot\|)$ be a normed space and $\emptyset\neq A$ and $\emptyset\neq B$ be strictly convex and disjoint subsets of $X$. Then there exists at most one BAP relative to $(A,B)$, and there exists a unique such BAP if, in addition, $\dist(A,B)$ is attained.
\end{cor}

\begin{proof}
This follows from Theorem \bref{thm:IntersectionStrictlyConvexSets} by letting $m:=n:=1$. 
\end{proof}

\begin{remark}\label{rem:StilesPai}
\begin{enumerate}[(i)]
\item\label{item:Luo} Corollary \bref{cor:StrictlyConvexSets} essentially appears in \cite[Theorem 4.1]{Luo2014jour}. We obtained Corollary \bref{cor:StrictlyConvexSets} more than a year before becoming aware of \cite[Theorem 4.1]{Luo2014jour}. See \cite[Theorems 4.2 and 4.3]{Luo2014jour} for additional related results. 
\item \label{item:Stiles} In \cite[Theorem 1.1]{Stiles1965b-jour} it is claimed that if $A$ and $B$ are two nonempty and disjoint subsets of a normed space $(X,\|\cdot\|)$ and if either $A$ or $B$ is strictly convex, then there exists at most one BAP relative to $(A,B)$ (this was formulated in the following somewhat different manner: ``the distance between $A$ and $B$ is attained at most at one point''). This claim is false, as Example \bref{ex:RectangleEllipseLinfty} below shows. The main mistake in \cite[Proof of Theorem 1.1]{Stiles1965b-jour} is the implicit assumption that the line segment $[P_B(x),P_B(y)]$ mentioned there is non-degenerate, and this is not necessarily true if $X$ is not strictly convex even if $B$ is strictly convex: again, see Example \bref{ex:RectangleEllipseLinfty}. We also note that \cite[Proof of Theorem 1.1]{Stiles1965b-jour} suffers from other issues, such as the unproven claim that if $A$ and $B$ are disjoint and if $(a,b)$ is a BAP relative to $(A,B)$, then $a\in\partial A$ and $b\in \partial B$ (this claim is true but requires a proof, as we showed in Lemma \bref{lem:a_0b_0Boundary} above) and the somewhat ambiguous notations $P_B(x)$ and $P_B(y)$ (while it is clear from \cite[Proof of Theorem 1.1]{Stiles1965b-jour} that both $(x,P_B(x))$ and $(y,P_B(y))$ are  BAPs with respect to $(A,B)$, when presenting  the operator of best approximation projection onto $B$ one needs to consider the issues of existence and uniqueness of this operator, and this has not been done in \cite[Proof of Theorem 1.1]{Stiles1965b-jour}). 

\item \label{item:PaiUniqueness} The sufficient condition (which is also necessary if both subsets are assumed to be closed) mentioned in Proposition \bref{prop:UniquenessInNormedSpace} for the at most uniqueness of a BAP relative to two nonempty, convex and disjoint subsets $A$ and $B$ of a general normed space is not easily verifiable. Other not easily verifiable sufficient (and necessary) conditions in a general normed space appear in \cite[Theorems 6.1 and 6.2]{Pai1974jour}, where the subsets are convex and have positive distance (this is not written explicitly, but follows from the proof which uses \cite[Theorem 4.4]{Pai1974jour} which assumes positive distance). For instance, in \cite[Theorem 6.2]{Pai1974jour} one needs to verify the non-existence of a continuous linear functional $g$ having the following properties (where we adopt our notation): its norm is 1, its real part attains a minimum over $A$ at two different points $a_0$ and $a_1$, its real part attains a minimum over $B$ at two different points $b_0$ and $b_1$, $g(a_0-b_0)=\|a_0-b_0\|$ and $g(a_1-b_1)=\|a_1-b_1\|$. 
\end{enumerate}
\end{remark}

\begin{thm}\label{thm:IntersectionStrictlyConvexNorm}
Given $m,n\in\N$, suppose that $A_1,A_2,\ldots,A_m$ and $B_1,B_2,\ldots,B_n$ are nonempty and convex subsets of a strictly convex normed space $(X,\|\cdot\|)$ such that $A:=\bigcap_{i=1}^m A_i$ and $B:=\bigcap_{j=1}^n B_j$ are nonempty and that $A\cap B= \emptyset$. If for each pair $(i,j)\in \{1,2,\ldots,m\}\times \{1,2,\ldots,n\}$ either:
\begin{enumerate}[(I)]
\item\label{item:A_iIsStrictlyConvex} $A_i$ is strictly convex, or 
\item\label{item:B_jIsStrictlyConvex} $B_j$ is strictly convex, or 
\item\label{ParallelIntervalsA_iB_j} there does not exist a pair of two nondegenerate and strictly parallel intervals such that one of them is contained in $\partial A_i$ and the other is contained in $\partial B_j$, 
\end{enumerate}
then there exists at most one BAP with respect to $(A,B)$. If, in addition, $\dist(A,B)$ is attained, then there exists a unique BAP with respect to $(A,B)$.
\end{thm}

\begin{proof}
By our assumption both $A$ and $B$ are convex and disjoint. If $\dist(A,B)$ is not attained, then obviously there exists at most one ( in fact, zero) BAP with respect  to $(A,B)$. Otherwise, there exists at least one BAP $(a_0,b_0)$ with respect to $(A,B)$. Suppose by way of contradiction that there exists another BAP $(a_1,b_1)\neq (a_0,b_0)$ relative to $(A,B)$. Then Lemma \bref{lem:parallel} implies that $[a_0,a_1]\subseteq \partial A$ and $[b_0,b_1]\subseteq \partial B$,  that both $[a_0,a_1]$ and  $[b_0,b_1]$ are nondegenerate, and that $[a_0,a_1]$ and $[b_0,b_1]$ are strictly parallel. As is well known, and can easily be proved (see, for instance \cite{BoundaryIntersection2016misc} for the case of two  subsets; the case of any finite number of subsets follows immediately by induction from the case of two subsets), the boundary of a finite intersection of subsets is contained in the union of the boundaries of the subsets which induce the intersection. Therefore,  $\partial A\subseteq \bigcup_{i=1}^m \partial A_i$. Hence, $[a_0,a_1]\subseteq \bigcup_{i=1}^m \partial A_i$, and so for each $t\in [0,1]$ there is at least one index $\phi(t)\in \{1,2,\ldots,m\}$ such that $a(t)=a_0+t(a_1-a_0)\in \partial A_{\phi(t)}$. Then $\phi$ is a function from $[0,1]$ to $\{1,2,\ldots,m\}$, and hence  we have $[0,1]=\cup_{i=1}^m\phi^{-1}(i)$. 

If $\phi^{-1}(i)$ is finite for each $i\in \{1,2,\ldots,m\}$, then so is the finite union $\cup_{i=1}^m\phi^{-1}(i)$, namely $[0,1]$ is finite, a contradiction. Hence $\phi^{-1}(i)$  is infinite for some $i\in \{1,2,\ldots,m\}$, i.e.,  there is an infinite subset $T_i\subseteq [0,1]$ such that $a(t)\in\partial A_i$ for each $t\in T_i$. In particular, there are three points $t_1<t_2<t_3$ in $T_i$, and since $a(t)=a_0+t(a_1-a_0)$ for all $t\in [0,1]$, the points $a_{t_1}$, $a_{t_2}$, and $a_{t_3}$ are three distinct points in $\partial A_i$. In addition, these points are contained in $[a_0,a_1]$, and since $t_1<t_2<t_3$ they satisfy $a_{t_2}\in [a_{t_1},a_{t_3}]$. We conclude from Lemma \bref{lem:ThreePointsIntervalBoundary} that $[a_{t_1},a_{t_3}]\subseteq \partial A_i$. Hence $A_i$ is not strictly convex, and so Assumption  \beqref{item:A_iIsStrictlyConvex} in the formulation of the theorem does not hold. Similarly, there are some $j\in \{1,2,\ldots,n\}$ and $t'_1$ and $t'_3$ in $[0,1]$ such that $[b_{t'_1},b_{t'_3}]$ is a nondegenerate interval contained in $[b_0,b_1]\cap \partial B_j$. Hence $B_j$ is not strictly convex, and so Assumption  \beqref{item:B_jIsStrictlyConvex} in the formulation of the theorem does not hold.  Since $[a_0,a_1]$ is strictly parallel to $[b_0,b_1]$ and since $[a_{t_1},a_{t_3}]\subseteq [a_0,a_1]$ and $[b_{t'_1},b_{t'_3}]\subseteq [b_0,b_1]$, it follows that  $[a_{t_1},a_{t_3}]$ is a nondegenerate interval which is contained in $\partial A_i$ and is strictly parallel to the nondegenerate interval $[b_{t'_1},b_{t'_3}]$ which is contained in $\partial B_j$, and this shows that also Assumption  \beqref{ParallelIntervalsA_iB_j} in the formulation of the theorem does not hold. 

We conclude that none of the Assumptions \beqref{item:A_iIsStrictlyConvex}--\beqref{ParallelIntervalsA_iB_j} in the formulation of the theorem holds, a contradiction. Consequently, the assumption that there exists more than one BAP with respect to $(A,B)$ cannot hold, namely, there exists a unique BAP with respect to $(A,B)$.
\end{proof}

From Theorem \bref{thm:IntersectionStrictlyConvexNorm} with $m:=n:=1$ we obtain the following corollary. 

\begin{cor}\label{cor:StrictlyConvexNorm}
Suppose that $A$ and $B$ are two nonempty, convex and disjoint subsets of a strictly convex normed space $(X,\|\cdot\|)$. If either 
\begin{enumerate}[(I)]
\item $A$ is strictly convex, or 
\item $B$ is strictly convex, or 
\item\label{item:ParallelBoundaries} there does not exist any pair of two nondegenerate and strictly parallel intervals having the property that one of them is contained in the boundary of $A$ and the other is contained in the boundary of $B$, 
\end{enumerate}
then there exists at most one BAP with respect to $(A,B)$. If, in addition, $\dist(A,B)$ is attained, then there exists a unique BAP with respect to $(A,B)$.
\end{cor}

\begin{remark}\label{rem:StrictConvexNormedSpace}
\begin{enumerate}[(i)]
\item The strict convexity of the norm in Theorem \bref{thm:IntersectionStrictlyConvexNorm} and Corollary \bref{cor:StrictlyConvexNorm} is essential for uniqueness (if not  both $A$ and $B$ are strictly convex): see Examples \bref{ex:RectangleEllipseLinfty} and \bref{ex:NonParallelLinfty} below for counterexamples. 

\item Theorem \bref{thm:IntersectionStrictlyConvexNorm} presents a simple to understand sufficient, but not necessary, geometric condition for the at most uniqueness of a BAP. For using this condition in practice in Case \beqref{item:ParallelBoundaries} one needs to analyze the structure of the boundaries of the sets $A_i$ and $B_j$, $(i,j)\in \{1,2,\ldots,m\}\times\{1,2,\ldots,n\}$ and to compare the orientation of line segments located in $\partial A_i$ to lines segments located in $\partial B_j$. See Section \bref{sec:Examples} for a few illustrations. This task can be demanding from the computational point of view, but again, it does offer a practical sufficient condition for uniqueness, and it also enables one to be aware of non-uniqueness which can occur due to parallel boundary intervals. 

\item Without imposing the non-parallelism assumption in Theorem \bref{thm:IntersectionStrictlyConvexNorm} and Corollary \bref{cor:StrictlyConvexNorm}, there can be counterexamples. Indeed, consider for instance Example \bref{ex:EllipseCylinder} and Figure \bref{fig:EllipseCylinder}, where the boundaries of both $A$ and $B$ contain non-degenerate parallel intervals which  constitute a strictly non-degenerate BAP of intervals with respect to $(A,B)$, and so there are (infinitely) many BAPs relative to $(A,B)$ in this case. On the other hand, uniqueness can hold even without the non-parallelism assumption, as Example \bref{ex:TwoD} and Figure \bref{fig:TwoD} show. Hence one can think about the non-parallelism condition as being analogous to the non-singularity condition on a smooth operator $T:U\to X$ acting from an open subset $U$ of a finite-dimensional space $X$ to $X$: if the Jacobian determinant of the derivative (that is, of the Jacobian matrix) of $T$ is not-vanishing at some point $x_0\in U$, then the equation $T(x)=y_0:=T(x_0)$ has a unique solution (i.e., $x=x_0$) in a neighborhood of $x_0$, but there might be cases in which this equation has a unique solution even if $\det(T'(x_0))=0$. In this connection, we note that asymptotic parallel lines do not necessarily lead to many BAPs: indeed, as the counterexample in Remark \bref{rem:Existence}\beqref{item:NoBAP} shows, they can lead to the non-existence of any BAP at all.

\item\label{item:A-A_B-B} Theorem \bref{thm:IntersectionStrictlyConvexNorm}\beqref{item:ParallelBoundaries} significantly generalizes \cite[Theorem 3.1]{Pai1974jour} which says that if $X$ is a uniformly convex Banach space (actually strict convexity is sufficient), $A$ and $B$ are closed and convex and $(A-A)\cap (B-B)=\{0\}$, then there exists at most one BAP relative to $(A,B)$ (see also \cite[Theorem 1.1]{Narang1984jour} for a generalization of this theorem to strictly convex metric linear spaces). Indeed, we can assume that $A\cap B=\emptyset$, since otherwise everything is trivial. Suppose that the above-mentioned condition holds and assume, for a contradiction, that Theorem \bref{thm:IntersectionStrictlyConvexNorm}\beqref{item:ParallelBoundaries} does not hold, namely that there are nondegenerate and strictly parallel intervals $[a_1,a_2]\subseteq\partial A$ and $[b_1,b_2]\subseteq\partial B$. Then either $u:=(a_2-a_1)/\|a_2-a_1\|$ and $v:=(b_2-b_1)/\|b_2-b_1\|$ are equal, or $u=-v$. Assume that the first case holds: the proof in the second case is similar. Let $r:=\min\{\|a_2-a_1\|,\|b_2-b_1\|\}$. Then $r>0$, $a_1+ru\in A$, $b_1+rv\in B$, and $ru=(a_1+ru)-a_1\in A-A$, $ru=rv=(b_1+rv)-b_1\in B-B$. Hence $ru$ is a nonzero vector in $(A-A)\cap (B-B)$, a contradiction which proves the assertion. We also note that the condition $(A-A)\cap (B-B)=\{0\}$ is frequently violated: indeed, just consider the case where both $A$ and $B$ have nonempty interior, as in the case of Figure  \bref{fig:PolygonDrop}: in this case Theorem \bref{thm:IntersectionStrictlyConvexNorm}\beqref{item:ParallelBoundaries} holds but there are $r>0$, $a\in A$ and $b\in B$ such that the open balls of radius $r$ and centers $a$ and $b$, respectively, are contained in $A$ and $B$, respectively, and hence, given an arbitrary unit vector $u\in X$, we have $a':=a+0.5ru\in A$, $b':=b+0.5ru\in B$, and $0\neq 0.5ru=a'-a=b'-b\in (A-A)\cap (B-B)$.
\end{enumerate}
\end{remark}

We finish this section with the following proposition which  generalizes \cite[Theorem 4, p. 51]{Xu1983jour}, which says that if $(X,\|\cdot\|)$ is a strictly convex Banach space, and $A$ and $B$ are closed and convex subsets of $X$, and if $(a_0,b_0)$ is a BAP relative to $(A,B)$, then  $(a_0,b_0)$ is the unique BAP relative to $(A,B)$ if and only if  $(A-a_0)\cap (B-b_0)=\{0\}$. Our proof is partly inspired by the proof of \cite[Theorem 4, p. 51]{Xu1983jour} and also corrects a minor issue which appears there. 
 
\begin{prop}\label{prop:UniquenessExistence}
Suppose that $(X,\|\cdot\|)$ is a normed space, that $A$ and $B$ are two subsets in $X$, and that $(a_0,b_0)$ is a BAP relative to $(A,B)$. If $(a_0,b_0)$ is the unique BAP relative to $(A,B)$, then  $(A-a_0)\cap (B-b_0)=\{0\}$. Conversely, if  $(X,\|\cdot\|)$ is strictly convex, $B-A$ is convex, and $(A-a_0)\cap (B-b_0)=\{0\}$, then $(a_0,b_0)$ is the unique BAP relative to $(A,B)$. In particular, if  $(X,\|\cdot\|)$ is strictly convex and $B-A$ is convex (as happens, in particular, when both $A$ and $B$ are convex), then $(a_0,b_0)$ is the unique BAP relative to $(A,B)$ if and only if $(A-a_0)\cap (B-b_0)=\{0\}$.
\end{prop}
\begin{proof}
Suppose that $(a_0,b_0)$ is the unique BAP relative to $(A,B)$ and assume, by way of contradiction, that $(A-a_0)\cap (B-b_0)\neq\{0\}$. Then there are $a_1\in A$, $b_1\in B$ and $0\neq v\in X$ which satisfy $a_1-a_0=v=b_1-b_0$. Therefore $b_0-a_0=b_1-a_1$, and hence $\|b_1-a_1\|=\|b_0-a_0\|=\dist(A,B)$. Thus $(a_1,b_1)$ is a BAP relative to $(A,B)$. On the other hand, since $v\neq 0$, it follows from the equalities $a_1-a_0=v=b_1-b_0$ that $a_1\neq a_0$ and $b_1\neq b_0$, and so $(a_0,b_0)\neq (a_1,b_1)$. Therefore $(a_0,b_0)$ and $(a_1,b_1)$ are two distinct BAPs relative to $(A,B)$, a contradiction to our assumption that $(a_0,b_0)$ is the unique BAP relative to $(A,B)$. This contradiction shows that $(A-a_0)\cap (B-b_0)=\{0\}$.

Conversely, suppose that $(X,\|\cdot\|)$ is strictly convex, that $(A-a_0)\cap (B-b_0)=\{0\}$ and that $B-A$ is convex. Assume, for a contradiction, that $(a_1,b_1)$ is a BAP relative to $(A,B)$ which satisfies $(a_1,b_1)\neq (a_0,b_0)$. It must be that $b_1-a_1\neq b_0-a_0$, otherwise $v:=a_1-a_0=b_1-b_0$, and so $v\in (A-a_0)\cap (B-b_0)=\{0\}$, and hence $v=0$, $b_1=b_0$ and $a_1=a_0$, a contradiction to $(a_1,b_1)\neq (a_0,b_0)$. Since both $(a_0,b_0)$ and $(a_1,b_1)$ are BAPs relative to $(A,B)$, we have  $\dist(A,B)=\|b_1-a_1\|=\|b_0-a_0\|$. These equalities imply that $\dist(A,B)>0$, otherwise $\dist(A,B)=0$ and hence $a_0=b_0$ and $a_1=b_1$, a contradiction to $(a_0,b_0)\neq (a_1,b_1)$. The previous lines imply that  $c_1:=b_1-a_1$ and $c_0:=b_0-a_0$ are two different vectors which belong to $B-A$ and are located on the boundary of the ball of radius $\dist(A,B)$ and center 0. Since $X$ is strictly convex, the middle point $c_{0.5}:=0.5c_0+0.5c_1$ of the line segment $[c_0,c_1]$ is located strictly inside this ball, namely $\|c_{0.5}\|<\dist(A,B)$. On the other hand, since we assumed that $B-A$ is convex and since $c_{0.5}$ is a convex combination of points from $B-A$, we have $c_{0.5}=b'-a'$ for some $b'\in B$ and $a'\in A$. Therefore $\|c_{0.5}\|=\|b'-a'\|\geq \dist(A,B)$ by the definition of $\dist(A,B)$, a contradiction to the inequality $\|c_{0.5}\|<\dist(A,B)$ which we showed before. This contradiction implies that  our initial assumption that there is a BAP $(a_1,b_1)$ relative to $(A,B)$ which is different from $(a_0,b_0)$ is false. Hence $(a_0,b_0)$ is the unique BAP relative to $(A,B)$. 
\end{proof}

\section{Existence of a best approximation pair}\label{sec:Existence}
In this section we present, in Theorem \bref{thm:Existence} below, many useful conditions which ensure the existence of a best approximating pair, and by doing this we significantly extend the known pool of such sufficient conditions. In particular, in some of these conditions we do not assume that $A$ and $B$ are convex. Most of these conditions are new, but some of them are known and we formulate them for the sake of completeness, and frequently provide some new information regarding them such as a new proof. In this connection, see Remark \bref{rem:Existence} below for various relevant comments, including counterexamples (Parts  \beqref{item:NoBAP}--\beqref{item:Nonreflexive}), a comparison with several published results (Parts \beqref{item:NoBAP}--\beqref{itemrem:HilbertPolytop}), and some extensions (Part \beqref{item:Dual}). 

\begin{thm}\label{thm:Existence}
Suppose that $A$ and $B$ are two nonempty subsets of a normed space $(X,\|\cdot\|)$. If at least one of the following conditions holds, then $\dist(A,B)$ is attained, namely there exists at least one BAP with respect to $(A,B)$:
\begin{enumerate}[(i)]
\item\label{item:Nonempty} {\bf (clear and well-known)} $A\cap B\neq \emptyset$;

\item\label{item:A-B_Proximinal0} {\bf (classic, explicit in \cite[pp. 58--59]{Pai1974jour})} $A-B$ is proximinal with respect to $\{0\}$. Equivalently, there is a minimal norm vector in $A-B$, that is, $\inf\{\|u\|\,|\,u\in A-B\}$ is attained;

\item\label{item:WeaklySeqCompactLocalWSC}  $A$ is weakly sequentially compact and $B$ is closed, convex and norm-locally weakly sequentially compact (see Definition  \bref{def:NormedLocalWeakSequentialCompact});

\item\label{item:CompactClosedConvexLocallyCompact} $A$ is compact and $B$ is closed, convex and locally compact;

\item\label{item:WeaklySequentiallyCompactProximinal} {\bf(\cite[Theorem 4]{Xu1988jour}):} $A$ is weakly sequentially compact and $B$ is convex and proximinal with respect to $A$;

\item \label{item:CompactProximinal} {\bf(\cite[Corollary 1]{Xu1988jour})} $A$ is compact and $B$ is proximinal with respect to $A$;

\item \label{item:BoundedlyCompactProximinal} $A$ is boundedly compact and $B$ is bounded and proximinal with respect to $A$;

\item \label{item:BoundedlyCompactBounded} {\bf(\cite[p. 1138]{Narang1976jour}, \cite[p. 322]{Xu1988jour})} $A$ and $B$ are boundedly compact and one of them is bounded;

\item\label{item:WeaklySequentiallyCompact} $A$ and $B$ are weakly sequentially compact;  

\item\label{item:compact} {\bf (classic,  \cite[p. 123]{Nicolescu1938jour})} $A$ and $B$ are compact; 
  
\item\label{item:DweaklyCompact}  For all closed balls $D$ in $X^2$ about the origin the intersection $D\cap (A\times B)$ is weakly sequentially compact, and either $A\cup B$ is bounded or it is unbounded and the coercivity condition  \beqref{eq:coercive} holds;

\item\label{item:WSC_a_k_Bounded} $X$ is a reflexive Banach space, $A$ and $B$ are weakly sequentially closed, and there is at least one distance minimizing sequence $((a_k,b_k))_{k\in\N}$ such that $(a_k)_{k\in\N}$  has a bounded subsequence;

\item\label{item:WSC-A_is_Bounded} $X$ is a reflexive Banach space, $A$ is weakly sequentially compact (alternatively, bounded and weakly sequentially closed), and $B$ is weakly sequentially closed;

\item\label{item:XuCor2} {\bf(\cite[Corollary 2]{Xu1988jour}):} $X$ is a reflexive Banach space, $A$ is bounded and weakly closed, and $B$ is closed and convex; 

\item\label{item:OneBounded}{\bf(\cite[Theorem 1.1]{Stiles1965b-jour})} $X$ is a reflexive Banach space, both $A$ and $B$ are convex and closed, and $A$ is bounded;

\item\label{item:B_D_coercive} $X$ is a reflexive Banach space, $A$ and $B$ are weakly sequentially closed, the union $A\cup B$ is unbounded, and the coercivity condition \beqref{eq:coercive} holds;

\item\label{item:ClosedConvexCoercive}  $X$ is a reflexive Banach space, both $A$ and $B$ are convex and closed, $A\cup B$ is unbounded and the coercivity condition \beqref{eq:coercive} holds;

\item\label{item:A-B_is_WSC} $X$ is a reflexive Banach space and $A-B$ is weakly sequentially closed.

\item\label{item:A-B_closed} $X$ is a reflexive Banach space, $A$ and $B$ are convex, and $A-B$ is closed;

\item \label{item:SumWeaklyClosedCompact} $X$ is a reflexive Banach space, $A=\wt{A}+\wh{A}$ and $B=\wt{B}+\wh{B}$, where both $\wt{A}$ and $\wt{B}$ are weakly sequentially compact, and $\wh{A}-\wh{B}$ is weakly sequentially closed;

\item\label{item:AisComplementedPi(B)Closed} $X$ is a reflexive Banach space, $A$ is a closed affine subspace with a closed linear part $\wt{A}$ which is complemented by a closed linear subspace $F$, and $B$ is an affine subspace with a linear part $\wt{B}$ such that $\Pi_F(\wt{B})$ is closed;

\item\label{item:B_finite_dim} $X$ is a reflexive Banach space, $A$ is a closed affine subspace, and $B$ is a finite-dimensional affine subspace;

\item \label{item:A_finite_codim} $X$ is a reflexive Banach space, $A$ is a closed affine  subspace of finite codimension, and $B$ is an affine subspace;

\item \label{item:DualWeakStar} $X$ is the dual of a separable Banach space $Y$ (e.g., $X=L_{\infty}(\Omega)$, $Y=L_1(\Omega)$, where $\Omega$ is a Lebesgue measurable subset of a Euclidean space), $A$ is weak-star sequentially compact, and $B$ is weak-star sequentially closed;

\item\label{item:HilbertPolytop} {\bf (implicit in  \cite{BauschkeBorweinLewis1997inproc})} $X$ is a real Hilbert space  and both $A$ and $B$ are polyhedral;

\item\label{item:VoronoiHilbert} $X$ is a Hilbert space,  $A\subseteq X$ is weakly sequentially closed, $p\in X\backslash A$, and $B$ is the Voronoi cell of $P:=\{p\}$ with respect to $A$, i.e., $B:=\{z\in X\,|\,\|z-p\|\leq d(z,A)\}$;

\item \label{item:Hyperparaboloid} $X$ is a Hilbert space, $A$ is a closed hyperplane, $p\in X\backslash A$, and $B$ is the full hyperparaboloid induced by $p$ and $A$, that is, the set of all points in $X$ whose distance to $p$ is not greater than their distance to $A$;  

\item\label{item:Hypercylinders} $X$ is a reflexive Banach space and both $A$ and $B$ are generalized hypercylinders with weakly sequentially compact bases;

\item \label{item:FiniteDimAffine} $A$ and $B$ are finite-dimensional affine subsapces;

\item\label{item:FiniteDimCoercive} $X$ is finite-dimensional, $A$ and $B$ are closed, and either $A\cup B$ is bounded or $A\cup B$ is unbounded and the coercivity condition \beqref{eq:coercive} holds;

\item\label{item:VoronoiFiniteDim} $X$ is finite-dimensional, $A$ is closed, $\emptyset\neq P\subset X$ is bounded, and $B$ is the Voronoi cell of $P$ with respect to $A$, namely $B:=\{z\in X\,|\,d(z,P)\leq d(z,A)\}$;

\item\label{item:polytop} {\bf (\cite[Theorem 5]{CheneyGoldstein1959jour}, \cite[the Theorem on p. 209]{Willner1968jour})} $X$ is a finite-dimensional Euclidean space and both $A$ and $B$ are polyhedral.

\end{enumerate}
\end{thm}

\begin{proof}
In what follows $((a_k,b_k))_{k\in \N}$ is a distance minimizing sequence in $A\times B$. 
\begin{enumerate}[(i)]
\item Since $A\cap B\neq \emptyset$ there is $b:=a\in A\cap B$, and hence $0=\|a-b\|=\dist(A,B)$ and $(a,b)$ is a BAP relative $(A,B)$.

\item By our assumption there is some $v\in A-B$ such that $\|v\|=\inf\{\|u\|\,|\,u\in A-B\}$, and so there is a pair $(a,b)\in A\times B$ such that $\|a-b\|=\inf\{\|u\|\,|\,u\in A-B\}$. It is immediate to verify that  $\dist(A,B)=d(0,A-B)=\inf\{\|u\|\,|\,u\in A-B\}$. Therefore $\dist(A,B)=\|a-b\|$ and $(a,b)$ is a BAP relative to $(A,B)$.

\item Since $A$ is weakly sequentially compact  there is an infinite subset $N_1\subseteq \N$ and a point $a\in A$ which satisfy $a=(w)\lim_{k\to\infty, k\in N_1}a_k$. Since any weakly convergent sequence is bounded  \cite[II.3.27, p. 68]{DunfordSchwartz1958book}, it follows that $(a_k)_{k\in N_1}$ is bounded, and therefore, by Lemma \bref{lem:(Un)BoundedMinimizing}\beqref{item:BoundedOrUnbounded}, also $(b_k)_{k\in N_1}$ is bounded. Thus (Lemma \bref{lem:WeaklyBolzanoWeierstrass})  $b=(w)\lim_{k\to\infty, k\in N_2}b_k$ for some infinite subset $N_2\subseteq N_1$ and $b\in B$. Hence $(a,b)=(w)\lim_{k\to\infty, k\in N_2}(a_k,b_k)$ and Lemma \bref{lem:(Un)BoundedMinimizing}\beqref{item:WeaklyLSC} implies that $(a,b)$ is a BAP relative to $(A,B)$.

\item This is an immediate consequence of Part \beqref{item:WeaklySeqCompactLocalWSC} because a compact subset is sequentially compact and hence (strong convergence implies weak convergence) also  weakly sequentially compact, and a locally compact subset is locally sequentially compact and hence also normed locally weakly sequentially compact.

\item Since $A$ is weakly sequentially compact  there exist an infinite subset $N_1\subseteq \N$  and a point $a\in A$ such that $(w)\lim_{k\to\infty, k\in N_1}a_{k}=a$. We claim that $d(a,B)=\dist(A,B)$. Indeed, consider the function $g:X\to [0,\infty)$ defined for all $x\in X$ by $g(x):=d(x,B)$. As is well known, $g$ is continuous (even Lipschitz continuous \cite[p. 19]{Jameson1974book}), and it is also convex since $B$ is convex \cite[Examples 5.18(b), p. 66]{VanTiel1984book}. Hence $g$ is weakly lower semicontinuous \cite[Corollary 3.9, p. 61]{Brezis2011book}. In addition, since $(b_k)_{k\in\N}$ is in $B$, we have $d(a_{k},B)\leq \|a_{k}-b_{k}\|$ for all $k\in N_1$ by the definition of $d(a_{k},B)$. Hence $d(a,B)=g(a)\leq \liminf_{k\to\infty, k\in N_1}g(a_{k})=\liminf_{k\to\infty, k\in N_1}d(a_{k},B)\leq \liminf_{k\to\infty}\|a_{k}-b_{k}\|=\dist(A,B)$. On the other hand $\dist(A,B)\leq d(a,B)$  because $a\in A$. Thus $d(a,B)=\dist(A,B)$. Since $B$ is proximinal with respect to $A$ there is $b\in B$ such that $d(a,B)=\|a-b\|$. Therefore $\|a-b\|=\dist(A,B)$ and so $(a,b)$ is a BAP relative to $(A,B)$.

\item From the compactness of $A$ there are $a\in A$ and an infinite subset $N_1\subseteq \N$ such that $\lim_{k\to\infty, k\in N_1}a_k=a$. Since $B$ is proximinal with respect to $A$ there is $b\in B$ such that $\|a-b\|=d(a,B)$. Since $(b_k)_{k\in \N}$ is in $B$, we have $d(a_k,B)\leq\|a_k-b_k\|$ for all $k\in \N$. Hence, because the function $g:X\to\ [0,\infty)$ defined for all $x\in X$ by $g(x):=d(x,B)$ is continuous \cite[p. 19]{Jameson1974book}), we have $\|a-b\|=d(a,B)=\lim_{k\to\infty, k\in N_1}d(a_k,B)\leq \lim_{k\to\infty, k\in N_1}\|a_k-b_k\|=\dist(A,B)$, where the last equality is by the assumption that $((a_k,b_k))_{k\in \N}$ is a distance minimizing sequence. Therefore $\|a-b\|\leq \dist(A,B)$, and obviously $\dist(A,B)\leq \|a-b\|$ since $(a,b)\in A\times B$. Thus $(a,b)$ is a BAP relative to $(A,B)$.

\item Since $B$ is bounded, so is $(b_k)_{k\in \N}$. Hence Lemma \bref{lem:(Un)BoundedMinimizing}\beqref{item:BoundedOrUnbounded} ensures that $(a_k)_{k\in\N}$ is bounded too. Let $C$ be a closed ball which contains both $(a_k)_{k\in \N}$ and $(b_k)_{k\in \N}$. Since $A$ is boundedly compact, $A\cap C$ is compact. Hence there are $a\in A\cap C$ and an infinite subset $N_1\subseteq \N$ such that $\lim_{k\to\infty, k\in N_1}a_k=a$. From now on we continue word for word as in the proof of Part \beqref{item:CompactProximinal} and conclude the existence of a BAP $(a,b)$ relative to $(A,B)$. 

\item Suppose that $A$ is bounded. The proof is similar if $B$ is bounded. Then $(a_k)_{k\in\N}$ is bounded, and hence, as follows from Lemma \bref{lem:(Un)BoundedMinimizing}\beqref{item:BoundedOrUnbounded}, also $(b_k)_{k\in\N}$ is bounded. Thus there is a closed ball $C$ such that both $(a_k)_{k\in\N}$ and $(b_k)_{k\in\N}$ are in $C$, and since both $A$ and $B$ are boundedly compact, the intersections $A\cap C$ and $B\cap C$ are compact. Thus $(A\cap C)\times (B\cap C)$ is a compact subset of $X^2$ which contains  $((a_k,b_k))_{k\in\N}$, and so there are  $(a,b)\in (A\cap C)\times (B\cap C)$ and an infinite subset $N_1\subseteq \N$ such that $\lim_{k\to\infty, k\in N_1}(a_k,b_k)=(a,b)$. Since the norm is continuous and since $((a_k,b_k))_{k\in\N}$ is a distance minimizing sequence, we have $\|a-b\|=\lim_{k\to\infty, k\in N_1}\|a_k-b_k\|=\dist(A,B)$, and so $(a,b)$ is a BAP relative to $(A,B)$. 

\item Since $A$ and $B$ are nonempty and weakly sequentially compact, so is their product $A\times B$, and so there is some $(a,b)\in A\times B$ which is the weak limit of a subsequence of $((a_k,b_k))_{k\in\N}$, that is $(a,b)=(w)\lim_{k\to\infty, k\in N_1}(a_k,b_k)$ for some infinite subset $N_1\subseteq \N$. Hence Lemma \bref{lem:(Un)BoundedMinimizing}\beqref{item:WeaklyLSC} implies that $(a,b)$ is a BAP relative to $(A,B)$.

\item  This is a consequence of Part \beqref{item:WeaklySequentiallyCompact} because any compact set is also sequentially compact and hence (strong convergence implies weak convergence)  weakly sequentially compact. Alternatively, one can show directly, using the continuity of the norm, that any accumulation point of $((a_k,b_k))_{k\in\N}$ (which exists because of the compactness of $A\times B$)  is a BAP relative to $(A,B)$.

\item Since either $A\cup B$ is bounded, or $A\cup B$ is unbounded and the coercivity condition \beqref{eq:coercive} holds, Lemma \bref{lem:(Un)BoundedMinimizing}\beqref{item:BoundedSequences} implies that $(a_k)_{k\in\infty}$ and $(b_k)_{k\in\N}$ are bounded. Hence  $((a_k,b_k))_{k\in\N}$ is contained in some closed ball $D$ of $X^2$ about the origin. Since $((a_k,b_k))_{k\in\N}$ is contained in $A\times B$, we conclude that $((a_k,b_k))_{k\in\N}$ is contained in  $C:=D\cap (A\times B)$, which is a weakly sequentially compact subset by the assumption in the formulation of this part. Hence there is a pair $(a,b)\in C$ and an infinite subset $N_1\subseteq \N$ such that $(a,b)=(w)\lim_{k\to\infty, k\in N_1}(a_k,b_k)$. Consequently, Lemma \bref{lem:(Un)BoundedMinimizing}\beqref{item:WeaklyLSC} implies that $(a,b)$ is a BAP relative to $(A,B)$.

\item Let $((a_k,b_k))_{k\in \N}$ be a distance minimizing sequence with the property that $(a_k)_{k\in \N}$ has a bounded subsequence $(a_k)_{k\in N_1}$ for some infinite subset $N_1\subseteq \N$. Since $X$ is reflexive, any bounded sequence in it has a weakly convergent subsequence  \cite[Theorem II.3.28, p. 68]{DunfordSchwartz1958book}. Hence there is some $a\in X$ and an infinite subset $N_2$ of $N_1$ such that $a=(w)\lim_{k\to\infty, k\in N_2}a_k$. Since $A$ is weakly sequentially closed, we have $a\in A$. Since $(a_k)_{k\in N_2}$ is bounded, also $(b_k)_{k\in N_2}$ is bounded by Lemma \bref{lem:(Un)BoundedMinimizing}\beqref{item:BoundedOrUnbounded}. Hence the reflexivity of $X$ implies that there is some $b\in X$ and an infinite subset $N_3$ of $N_2$ such that $b=(w)\lim_{k\to\infty, k\in N_3}b_k$. Thus $(a,b)=(w)\lim_{k\to\infty, k\in N_3}(a_k,b_k)$ and $b\in B$ since $B$ is weakly sequentially closed. Hence Lemma \bref{lem:(Un)BoundedMinimizing}\beqref{item:WeaklyLSC} implies that $(a,b)$ is a BAP relative to $(A,B)$. 

\item Since $A$ is weakly sequentially compact, it must be bounded (otherwise there is some sequence $(x_k)_{k\in\N}$ in $A$ such that $\lim_{k\to\infty}\|x_k\|=\infty$, and hence $(x_k)_{k\in\N}$ cannot have a weakly convergent subsequence since any weakly convergent  subsequence is bounded \cite[II.3.27, p. 68]{DunfordSchwartz1958book}; thus not every sequence in $A$ has a convergent subsequence, in contradiction with the assumption that $A$ is weakly sequentially compact). In addition, as a weakly sequentially compact subset, $A$ is evidently weakly sequentially closed. Hence for every distance minimizing sequence $((a_k,b_k))_{k\in\N}$, the sequence $(a_k)_{k\in \N}$ is automatically bounded. Thus the assertion follows from Part \beqref{item:WSC_a_k_Bounded}. 

\item Since $A$ is weakly closed and bounded and since $X$ is reflexive, $A$ is weakly compact \cite[Corollary V.4.8, p. 415]{DunfordSchwartz1958book}, and hence weakly sequentially compact since any weakly compact subset of a normed space is weakly sequentially compact \cite[Corollary in Section 18A, p. 146]{Holmes1975book}. Since $B$ is closed and convex, it is weakly closed  \cite[Theorem V.3.13, p. 422]{DunfordSchwartz1958book}, and so weakly sequentially closed. Thus the assertion follows from  Part \beqref{item:WSC-A_is_Bounded}.

\item The result follows from either Part \beqref{item:XuCor2} or Part \beqref{item:WSC-A_is_Bounded} because any nonempty, closed and convex subset is weakly closed and hence weakly sequentially closed, and any nonempty, closed, convex and bounded  subset of a reflexive Banach space is weakly compact and hence weakly sequentially compact. 

\item Because $A\cup B$ is unbounded and the coercivity condition \beqref{eq:coercive} holds, we conclude from Lemma \bref{lem:(Un)BoundedMinimizing}\beqref{item:BoundedSequences} that $(a_k)_{k\in\infty}$ and $(b_k)_{k\in\N}$ are bounded for every distance minimizing sequence $((a_k,b_k))_{k\in\N}$. The assertion now follows from Part \beqref{item:WSC_a_k_Bounded}.

\item The assertion follows from Part \beqref{item:B_D_coercive} because any closed and convex subset of a Banach space is weakly closed and hence weakly sequentially closed. 

\item Since $\lim_{k\to\infty}\|a_k-b_k\|=\dist(A,B)<\infty$, if we denote $z_k:=a_k-b_k$ for every $k\in \N$, then $(z_k)_{k\in\N}$ is bounded and hence (because $X$ is reflexive) $(w)\lim_{k\to\infty, k\in N_1}z_k=z$ for some infinite subset $N_1\subseteq \N$ and some $z\in X$. Because $A-B$ is weakly sequentially closed and  $(z_k)_{k\in\N}$ is in $A-B$, we have $z\in A-B$. Thus $z=a-b$ for some $(a,b)\in A\times B$. In addition, since the norm is weakly sequentially lower semicontinuous  \cite[II.3.27, p. 68]{DunfordSchwartz1958book}, we have $\|a-b\|=\|z\|\leq\liminf_{k\to\infty, k\in N_1}\|z_k\|=\dist(A,B)$. Since $(a,b)\in A\times B$, we obviously have $\dist(A,B)\leq \|a-b\|$. Hence $\|a-b\|=\dist(A,B)$ and $(a,b)$ is a BAP relative to $(A,B)$.

\item By the assumptions on $A$ and $B$ we see that $A-B$ is closed and convex. Hence $A-B$ is weakly closed and therefore weakly sequentially closed.  Since $X$ is reflexive, the assertion follows from Part \beqref{item:A-B_is_WSC}.

\item Since $A=\wt{A}+\wh{A}$ and $B=\wt{B}+\wh{B}$, an immediate verification shows that $A-B=(\wt{A}-\wt{B})+(\wh{A}-\wh{B})$. Hence by Part \beqref{item:WSC-A_is_Bounded} it is sufficient to show that $\wt{A}-\wt{B}$ is weakly sequentially compact because we already assume that $\wh{A}-\wh{B}$ is weakly sequentially closed. This is immediate because both $\wt{A}$ and $\wt{B}$ are weakly sequentially compact (hence if $(x_k-y_k)_{k\in\N}$ is an arbitrary sequence in $\wt{A}-\wt{B}$ where $(x_k)_{k\in\N}$ is in $\wt{A}$ and $(y_k)_{k\in\N}$ is in $\wt{B}$, then we can find infinite subsets $N_2\subseteq N_1\subseteq \N$ and points $x\in \wt{A}$ and $y\in \wt{B}$ such that $x=(w)\lim_{k\to\infty, k\in N_1}x_k$ and $y=(w)\lim_{k\to\infty, k\in N_2}y_k$; thus $x-y=(w)\lim_{k\to\infty, k\in N_2}(x_k-y_k)$, namely $(x_k-y_k)_{k\in \N}$ has a subsequence which converges to a point in $\wt{A}-\wt{B}$). 

\item Our goal is to use Part \beqref{item:A-B_closed}. Since $A$ and $B$ are affine and hence convex, it remains  to show that $A-B$ is closed. We can write $A=p_1+\wt{A}$ and $B=p_2+\wt{B}$ for some $p_1,p_2\in X$ and linear subspaces $\wt{A}$ and $\wt{B}$ of $X$. Since $A-B=(p_1-p_2)+(\wt{A}-\wt{B})$, it is sufficient to show that $\wt{A}-\wt{B}$ is closed. 

We claim that $\wt{A}-\wt{B}=\wt{A}\oplus\Pi_F(\wt{B})$. Indeed, let $z\in \wt{A}-\wt{B}$ be arbitrary. Then $z=x-y$ for some $x\in \wt{A}$ and $y\in \wt{B}$. From our assumption that $X=\wt{A}\oplus F$ we can write $y=y_1+y_2$, where $y_1=\Pi_{\wt{A}}(y)\in \wt{A}$ and $y_2=\Pi_F(y)\in F$. Because $\wt{A}$, as a linear subspace, is closed under sums, we have $x-y_1\in \wt{A}$. In addition, $\Pi_F(\wt{B})=-\Pi_F(\wt{B})$ since $\Pi_F(\wt{B})$ is a linear subspace. Hence $-y_2=-\Pi_F(y)\in \Pi_F(\wt{B})$ and $x-y=(x-y_1)+(-y_2)\in \wt{A}+\Pi_F(\wt{B})$. Since $z\in \wt{A}-\wt{B}$ was arbitrary, we have $\wt{A}-\wt{B}\subseteq \wt{A}+\Pi_F(\wt{B})$. Because  $\Pi_F(\wt{B})\subseteq F$ and $\wt{A}\cap F=\{0\}$, we actually have $\wt{A}+\Pi_F(\wt{B})=\wt{A}\oplus \Pi_F(\wt{B})$. Now let $z\in \wt{A}\oplus \Pi_F(\wt{B})$ be arbitrary. Then $z=x+w$ for some (unique) $x\in \wt{A}$ and $w\in \Pi_F(\wt{B})$. Since  $w\in \Pi_F(\wt{B})$, there is some $y\in \wt{B}$ such that $w=\Pi_F(y)$.  Hence $z=x+w=(x-\Pi_{\wt{A}}(y))+(\Pi_{\wt{A}}(y)+\Pi_F(y))=(x-\Pi_{\wt{A}}(y))+y=(x-\Pi_{\wt{A}}(y))-(-y)$. Because $\wt{A}$ and $\wt{B}$ are linear subspaces, we have $x-\Pi_{\wt{A}}(y)\in \wt{A}$ and $-y\in\wt{B}$. Therefore $z\in \wt{A}-\wt{B}$. Since $z\in \wt{A}\oplus \Pi_F(\wt{B})$ was arbitrary, we have  $\wt{A}\oplus\Pi_F(\wt{B})\subseteq \wt{A}-\wt{B}$, as required. 

We claim that $\wt{A}\oplus \Pi_F(\wt{B})$ is a closed subset of $X$. Indeed, let $(z_k)_{k\in\N}$ be any convergent sequence in $\wt{A}\oplus\Pi_F(\wt{B})$, and let $z\in X$ be its limit. Then for all $k\in \N$, one has $z_k=x_k+w_k$  for some (unique) $x_k\in \wt{A}$ and $w_k\in \Pi_F(\wt{B})\in F$. Since $\wt{A}$ and $F$ are topologically  complemented in the Banach space $X$, the linear projection $\Pi_{\wt{A}}$ is continuous \cite[Theorems 13.1, 13.2, p. 94]{Conway1990book}. Hence $\lim_{k\to\infty}x_k=\lim_{k\to\infty}\Pi_{\wt{A}}(z_k)=\Pi_{\wt{A}}(z)\in \wt{A}$. Thus $w:=\lim_{k\to\infty}w_k=\lim_{k\to\infty}(z_k-x_k)=z-\Pi_{\wt{A}}(z)$. Because $(w_k)_{k\in\N}$ is in the closed subspace $\Pi_F(\wt{B})$, its limit $w$ is in $\Pi_F(\wt{B})$. Hence $z=\Pi_{\wt{A}}(z)+w\in \wt{A}\oplus \Pi_F(\wt{B})$ and $\wt{A}\oplus \Pi_F(\wt{B})$ is closed. Since $\wt{A}-\wt{B}=\wt{A}\oplus\Pi_F(\wt{B})$, also $\wt{A}-\wt{B}$ is closed, as required. 

\item Since the linear part $\wt{A}$ of $A$ is closed (because so is $A$), and since we assume that the linear part $\wt{B}$ of $B$ is finite dimensional, we conclude from \cite[Proposition 20.1, p. 195]{Jameson1974book} that $\wt{A}+\wt{B}$ is closed. Since obviously $\wt{B}=-\wt{B}$ because $\wt{B}$ is a linear subspace, we see that $\wt{A}-\wt{B}=\wt{A}+\wt{B}$ is closed, and hence so is its translated copy $A-B$. Therefore the assertion follows from Part \beqref{item:A-B_closed}. 

\item By our assumption $X=\wt{A}\oplus F$ for some finite dimensional linear subspace $F$. Therefore  $\Pi_F(\wt{B})$, which is a linear subspace of $F$,  is also finite dimensional and hence closed \cite[p. 196]{Jameson1974book}. The assertion now follows from Part \beqref{item:AisComplementedPi(B)Closed}.

\item Since $A$ is weak-star  sequentially compact, it is evidently weak-star sequentially closed. Moreover, $A$ must be bounded, otherwise there is some sequence $(x_k)_{k\in\N}$ in $A$ such that $\lim_{k\to\infty}\|x_k\|=\infty$, and hence $(x_k)_{k\in\N}$ cannot have a weak-star convergent subsequence since any weak-star convergent  subsequence is bounded \cite[Proposition 3.13(iii), p. 63]{Brezis2011book}; therefore not every sequence in $A$ has a convergent subsequence in the weak-star topology, in contradiction with the assumption that $A$ is weak-star sequentially compact.  Hence $A$ is bounded and thus for every distance minimizing sequence $((a_k,b_k))_{k\in\N}$, the sequence $(a_k)_{k\in \N}$ is automatically bounded. Therefore, as follows from Lemma \bref{lem:(Un)BoundedMinimizing}\beqref{item:BoundedOrUnbounded}, also $(b_k)_{k\in\N}$ is bounded.

Since $X$ is the dual of a separable Banach space $Y$, every bounded sequence in $X$ has a weak-star convergent subsequence: see, for instance, \cite[Corollary 3.30, p. 76]{Brezis2011book}. Hence there is some $a\in X$ and an infinite subset $N_1$ of $\N$ such that $a=(w^*)\lim_{k\to\infty, k\in N_1}a_k$, and there is an infinite subset $N_2$ of $N_1$ and $b\in X$ such that $b=(w^*)\lim_{k\to\infty, k\in N_2}b_k$. Thus $a-b=(w^*)\lim_{k\to\infty, k\in N_2}(a_k-b_k)$. Since $A$ and $B$ are weak-star sequentially closed, we have $a\in A$ and $b\in B$. Since the norm of $X$, which is the dual of the norm of $Y$, is weak-star lower semicontinuous \cite[Proposition 3.13(iii), p. 63]{Brezis2011book}, and since  $((a_k,b_k))_{k\in\N}$ is a distance minimizing sequence, we have $\|a-b\|\leq \liminf_{k\in N_2}\|a_k-b_k\|=\dist(A,B)$. But $\dist(A,B)\leq\|a-b\|$ because $(a,b)\in A\times B$. Hence $\|a-b\|=\dist(A,B)$ and $(a,b)$ is a BAP relative to $(A,B)$. 

\item According to \cite[Corollary 3.4.8]{BauschkeBorweinLewis1997inproc}, the infimum $\sigma:=\inf\{\|z-P_BP_Az\|\,|\, z\in X\}$ is attained at some $b\in X$, namely $\sigma=\|b-P_BP_Ab\|$, where $P_A$ is the orthogonal projection on $A$ and $P_B$ is the orthogonal projection on $B$, which are well defined since $A$ and $B$ are nonempty, closed and convex. According to \cite[Corollary 4.4.3, Fact 4.4.4 and Remark 4.4.6]{BauschkeBorweinLewis1997inproc}, since $\sigma$ is attained, one has $\sigma=0$.  Thus $\|b-P_BP_Ab\|=0$, namely $b$ is a fixed point of $P_BP_A$, and, in particular, $b\in B$. But  according to \cite[Fact 5.1.4(i)]{BauschkeBorweinLewis1997inproc}, which is actually \cite[Theorem 2]{CheneyGoldstein1959jour}, any fixed point $z$ of $P_BP_A$ satisfies $d(z,A)=\dist(A,B)$. Since $d(z,A)=\|z-P_Az\|$ by the definition of $P_A$, if we let $z:=b$ and $a:=P_Ab$, then $a\in A$ and  $\dist(A,B)=\|b-a\|$, that is, $(a,b)$ is a BAP relative to $(A,B)$. 

\item Since $B=\cap_{a\in A}H(p,a)$, where $H(p,a):=\{z\in X\,|\, \|z-p\|\leq \|z-a\|\}$, it follows that $B$ is an  intersection of closed real halfspaces and hence closed and convex (since simple arithmetic shows that $H(p,a)=\{z\in X\,|\,\Rep \langle z,a-p\rangle  \leq  0.5(\|a\|^2-\|p\|^2)\}$, and this set is a real halfspace because the relation $p\notin A$ implies that $p\neq a$ for all $a\in A$). Thus $B$ is weakly closed and thus weakly sequentially closed. Therefore if $B$ is bounded, then the assertion follows from Part \beqref{item:WSC-A_is_Bounded} (where $B$ and $A$ are interchanged there). Otherwise, $B$ is unbounded and so is $A\cup B$. We claim that \beqref{eq:coercive} holds. Indeed, let $\mu>0$ be arbitrary and denote $\rho:=3(\mu+\|p\|)$. Let $(x,y)$ be an arbitrary pair in  $A\times B$ which satisfies $\|(x,y)\|>\rho$. Either $\|y\|>\mu+\|p\|$ or $\|y\|\leq \mu+\|p\|$. In the first case the  relations $y\in B$, $x\in A$ and the triangle inequality imply that $\mu<\|y\|-\|p\|\leq \|y-p\|\leq d(y,A)\leq \|y-x\|$. In the second case we must have $\|x\|>2(\mu+\|p\|)$ because otherwise $\|(x,y)\|^2=\|x\|^2+\|y\|^2\leq 5(\mu+\|p\|)^2<9(\mu+\|p\|)^2=\rho^2$, a contradiction to what we assumed on $(x,y)$. Thus the triangle inequality and the inequalities $\|y\|\leq \mu+\|p\|$ and $\|x\|>2(\mu+\|p\|)$ imply that $\mu\leq \mu+\|p\|<\|x\|-\|y\|\leq \|y-x\|$. Therefore $\mu<\|x-y\|$ whenever $(x,y)\in A\times B$ satisfies $\|(x,y)\|>\rho$. Since $\mu$ was an arbitrary positive number, the definition of the limit implies that \beqref{eq:coercive} holds. Consequently, Part \beqref{item:B_D_coercive} implies the existence of a BAP with respect to $(A,B)$. 

\item $B$ is nothing but the Voronoi cell of $P:=\{p\}$ with respect to $A$, i.e., $B:=\{z\in X\,|\, \|z-p\|\leq d(z,A)\}$. Since $A$ is closed and convex, it is weakly closed and hence weakly sequentially closed. Thus the assertion follows from Part \beqref{item:VoronoiHilbert} 

\item By our assumption $A=(C_1-c_1)+L_1$ and $B=(C_2-c_2)+L_2$ for two real lines $L_1$ and $L_2$, and two weakly sequentially compact sets $C_1\subseteq X$ and $C_2\subseteq X$ such that $L_i\cap C_i=\{c_i\}$ for some $c_i\in X$, $i\in \{1,2\}$. Since $L_i=u_i+\wt{L}_i$, where $u_i\in X$ and $\wt{L}_i$ is the linear part of $L_i$, $i\in\{1,2\}$, the linear part of $L_1-L_2$ is the linear subspace  $\wt{L}_1-\wt{L}_2$. This is a finite-dimensional linear subspace (whose dimension is at most 2) and hence a closed set \cite[p. 196]{Jameson1974book}. Thus $L_1-L_2$ is closed. Because $L_1-L_2$ is also convex, it is weakly closed and hence weakly sequentially closed. The assertion now follows from Part \beqref{item:SumWeaklyClosedCompact} since $C_1-c_1$ and $C_2-c_2$ are weakly sequentially compact as translations of weakly sequentially compact sets.

\item Since both $A$ and $B$ are finite-dimensional affine subspaces, so is their difference $A-B$. As is well known, the distance from any point $x$ in a normed space $X$ to a finite-dimensional affine subspace $F$ of $X$ is attained (this is immediate: let $(x_k)_{k\in\N}$ in $F$ satisfy $\lim_{k\to\infty}\|x-x_k\|=d(x,F)$; then $(x_k)_{k\in\N}$ is bounded, and so has a convergent subsequence since $F$ is finite-dimensional, and the limit $z$ of this subsequence satisfies $\|x-z\|=d(x,F)$ by the continuity of the norm). Hence $d(0,A-B)$ is attained and the assertion follows from Part \beqref{item:A-B_Proximinal0}.

\item Since any finite-dimensional normed space is a reflexive Banach space and since in finite-dimensional normed spaces a sequence converges weakly if and only if it converges strongly, the assertion follows from either Part \beqref{item:WSC-A_is_Bounded} (if $A\cup B$ is bounded) or Part \beqref{item:B_D_coercive} (if $A\cup B$ is unbounded).

\item The proof is somewhat similar to the proof of Part \beqref{item:VoronoiHilbert}, but because there are differences in the settings, some modifications are needed. First we observe that since the function $g:X\to \R$ defined by $g(z):=d(z,P)-d(z,A)$ for all $z\in X$ is continuous (even Lipschitz continuous) and $B$ is its zero-sublevel-set, it follows that $B$ is closed and hence weakly sequentially closed since $X$ is finite dimensional. We also observe that since $d(z,P)=d(z,\ol{P})$ for every $z\in X$ (as follows, for instance, from Lemma \bref{lem:DistClosures}),  we have $B=\{z\in X\,|\,d(z,\ol{P})\leq d(z,A)\}$. 

If $B$ is bounded, then the assertion follows from Part \beqref{item:WSC-A_is_Bounded} (where $B$ and $A$ are interchanged there). Otherwise, $B$ is unbounded and so is $A\cup B$. We claim that \beqref{eq:coercive} holds. Indeed, let $\mu>0$ be arbitrary. Since $P$ is bounded, so is $\ol{P}$, and there is some $r>0$ such that $\ol{P}$ is contained in the ball of radius $r$ about the origin. Denote $\rho:=3(\mu+r)$ and let $(x,y)$ be an arbitrary pair in  $A\times B$ which satisfies $\|(x,y)\|>\rho$. Either $\|y\|>\mu+r$ or $\|y\|\leq \mu+r$. Suppose that the first case holds. Since $\ol{P}$ is closed and $X$ is finite dimensional, there is some $p\in \ol{P}$ (hence $\|p\|\leq r$) such that $\|y-p\|=d(y,\ol{P})$ (this also follows from Part \beqref{item:WSC-A_is_Bounded}, where $A$ there is replaced by $\{y\}$ and $B$ there is replaced by $\ol{P}$). These facts, as well as the triangle inequality and the fact that $y\in B$, all imply that $\mu<\|y\|-r\leq \|y\|-\|p\|\leq \|y-p\|=d(y,\ol{P})\leq d(y,A)\leq \|y-x\|$. Now suppose that the second case holds, that is, $\|y\|\leq \mu+r$. It must be that $\|x\|>2(\mu+r)$ because otherwise $\|(x,y)\|^2=\|x\|^2+\|y\|^2\leq 5(\mu+r)^2<\rho^2$ by the definition of $\rho$, a contradiction to what we assumed on $(x,y)$. Hence the inequalities $\|y\|\leq \mu+r$ and $\|x\|>2(\mu+r)$, as well as the triangle inequality, imply that $\mu< \mu+r<\|x\|-\|y\|\leq \|y-x\|$. Thus $\mu<\|x-y\|$ for all $(x,y)\in A\times B$ which satisfies $\|(x,y)\|>\rho$. Since $\mu$ was an arbitrary positive number, the definition of the limit implies that \beqref{eq:coercive} holds. Hence Part \beqref{item:B_D_coercive} implies that there is a BAP relative to $(A,B)$. 

\item This is just a particular case of Part  \beqref{item:HilbertPolytop}. Alternatively, since $B$ is polyhedral, also $-B$ is polyhedral, and since the sum of finite-dimensional polyhedral sets is polyhedral by \cite[Corollary 19.3.2]{Rockafellar1970book} (see also \cite[Lemma 2]{Willner1968jour}), it follows that $A-B$ is polyhedral and hence closed because a polyhedral set is closed as an intersection of closed sets. The assertion now follows from Part  \beqref{item:A-B_closed}.
\end{enumerate}
\end{proof}

\begin{figure}[t]
\begin{minipage}[htb]{0.44\textwidth}
\begin{center}{\includegraphics[trim=100 670 330 70, clip=true, scale=0.57]{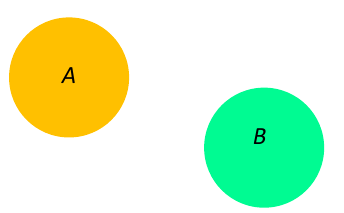}}
\end{center}
 \caption{Two open and disjoint discs in the Euclidean plane: no BAP (Remark \bref{rem:Existence}\beqref{item:NoBAP}).}
\label{fig:NoBAP-discs}
\end{minipage}
\hfill
\begin{minipage}[htb]{0.55\textwidth}
\begin{center}{\includegraphics[trim=120 350 110 180, clip=true, scale=0.39]{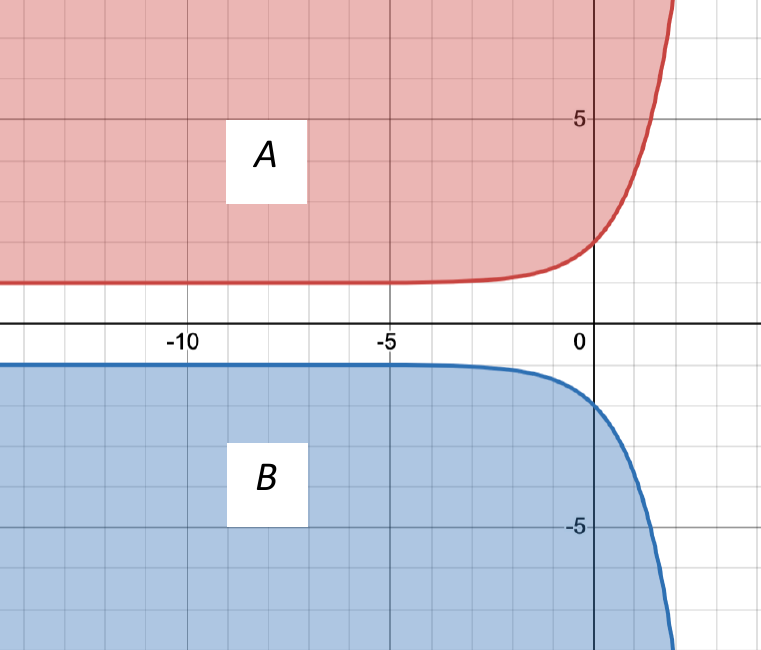}}
\end{center}
 \caption{Illustration of the second counterexample in Remark \bref{rem:Existence}\beqref{item:NoBAP}: two closed, strictly convex, boundedly compact and disjoint sets in the Euclidean plane where there is no BAP relative to them (here the coercivity condition \beqref{eq:coercive} does not hold).}
\label{fig:NoBAP-StrictlyConvexEuclidean}
\end{minipage}
\end{figure}
\begin{remark}\label{rem:Existence}
Here are a few comments related to Theorem \bref{thm:Existence}:

\begin{enumerate}[(i)]
\item \label{item:NoBAP} In general, existence of a BAP is not ensured even in very simple settings. For instance, just consider two open and disjoint line segments in the real line, such as $A=(0,2)$ and $B=(5,10)$, or two open and disjoint discs in the Euclidean plane (in the first case if we replace both sets with their closures, then the pair $(2,5)$ will be the unique BAP, in accordance with the combination of Corollary \bref{cor:StrictlyConvexNorm} and Theorem \bref{thm:Existence}\beqref{item:compact}). For a second counterexample, let $X$ be the Euclidean plane, $A:=\{(x_1,x_2)\in \R^2\,|\, x_2\geq e^{x_1}+1\}$ and $B:=\{(x_1,x_2)\in \R^2\,|\, x_2\leq -e^{x_1}-1\}$. Then $A$ and $B$ are closed, strictly convex and boundedly compact, but any $(a,b)\in A\times B$ satisfies $\|a-b\|>2=\dist(A,B)$, and so there is no BAP relative to $(A,B)$. See Figure \bref{fig:NoBAP-StrictlyConvexEuclidean}. Of course, in this case the coercivity condition \beqref{eq:coercive} does not hold. This example shows that  \cite[Theorem 2.3, p. 385]{Singer1970book}, which claims that there is a BAP relative to $(A,B)$ whenever $A$ and $B$ are boundedly compact and closed,  is incorrect, as observed before in \cite[p. 1138]{Narang1976jour} and later in \cite[p. 322]{Xu1988jour} using different  (counter)examples. Nevertheless, by adding to \cite[Theorem 2.3, p. 385]{Singer1970book} the assumption that either $A$ or $B$ is bounded, the assertion becomes correct, as was observed without a proof in both \cite[p. 1138]{Narang1976jour} and  \cite[p. 322]{Xu1988jour} and is proved in Part \beqref{item:BoundedlyCompactBounded}. A third counterexample appears in Part \beqref{item:Nonreflexive} below. A fourth counterexample appears in \cite[Theorem 5]{Lin1966jour}, which says that in every infinite-dimensional real normed space there exist closed, convex and linearly bounded subsets $A\neq\emptyset$ and $B\neq\emptyset$ such that there is no BAP relative to $(A,B)$. 

\item\label{item:Nonreflexive} While various parts of Theorem \bref{thm:Existence} do not require the space to be reflexive, the reflexivity assumption,  which is imposed in many parts of the theorem, is not a coincidence. Indeed, in any real non-reflexive Banach space there are very simple closed, convex and disjoint subsets $A$ and $B$ such that there is no BAP relative to them. In fact, $A$ can be taken to be a singleton and $B$ can be taken to be a closed hyperplane. This  property actually characterizes real non-reflexive Banach spaces \cite[p. 253]{Phelps1960jour}, \cite[Theorem 19C(a), p. 161]{Holmes1975book}, \cite[Theorem 1]{Lin1966jour}, a result which is attributed in \cite[p. 253]{Phelps1960jour} to James, based on his celebrated theorem (in its earlier version formulated for separable spaces) which says that a real Banach space is non-reflexive if and only if there exists a linear continuous functional which does not attain a maximum on the unit sphere of the space: see \cite[Theorem 5]{James1964jour} (which discusses only the difficult direction) and \cite[Theorem 2]{James1972jour}. Closely related characterizations of real reflexive Banach spaces appear in \cite[Theorems 3 and 4]{Lin1966jour}: e.g., \cite[Theorem 4]{Lin1966jour} says that a real Banach space is reflexive if and only if for all pairs $(A,B)$ of nonempty, closed, convex and bounded subsets $A$ and $B$ of $X$ there is a BAP relative to $(A,B)$.

\item\label{itemrem:Known}  Part \beqref{item:Nonempty} is, of course, well known and is mentioned in \cite[p. 434 in Section 5]{BauschkeBorwein1994jour} and \cite[Fact 2.3(v)]{BauschkeCombettesLuke2004jour}, in the context of Hilbert spaces.  Part \beqref{item:CompactClosedConvexLocallyCompact} generalizes \cite[the assertion after Theorem 3]{Xu1988jour} (which by itself generalizes \cite[Theorem 3]{SahneySingh1980jour}, a result whose proof suffers from issues), where there $X$ is restricted to be a  Banach space and $A$ (denoted by $G$ there) is also assumed to be convex (the proof of this assertion suffers from a gap, namely the convexity of $F$ there - denoted by $B$ in Part \beqref{item:CompactClosedConvexLocallyCompact} - is crucial for the existence of a convergent subsequence in $F$, but is omitted from the proof). Part \beqref{item:WeaklySequentiallyCompactProximinal} significantly generalizes  \cite[Theorem 3.1]{Pai1974jour} and also generalizes  \cite[Theorem (8), p. 345]{Kothe1969book} because any closed, convex and weakly locally compact subset of a normed space is proximinal with respect to the whole space \cite[Theorem 
(1), p. 343]{Kothe1969book}. Part \beqref{item:WeaklySequentiallyCompact} generalizes an assertion made in \cite[between Theorem 2.4 and Definition 2.5]{DigarKosuru2020jour} in a Banach space setting. Part \beqref{item:A-B_closed} generalizes \cite[Lemma 2.1(ii) and p. 434 in Section 5]{BauschkeBorwein1994jour} (see also \cite[Theorem 5.4.3]{BauschkeBorweinLewis1997inproc}) from the case where $X$ is a real Hilbert space and $A$ and $B$ are closed and convex. Parts \beqref{item:B_finite_dim}--\beqref{item:A_finite_codim} combined generalize  \cite[Facts 5.1(iii)]{BauschkeBorwein1994jour} from the case where $X$ is a real Hilbert space and $A$ and $B$ are closed and convex. Part \beqref{item:FiniteDimCoercive} extends related existence results in the Euclidean case, such as \cite[Corollary 4.16]{CaseiroFacasVicenteVitoria2019jour}, \cite[Proposition 2.2, Corollary 2.3]{FacasVicenteGoncalvesVitoria2014jour},\cite[Proposition 2.3]{GoncalvesFacasVicenteVitoria2015jour} (without giving explicit formulae as done there; note that there a BAP is referred to as ``the'' BAP, although there can be several other BAPs).

\item\label{item:ConvexBG} Part \beqref{item:CompactProximinal} is stated  without a proof in \cite[Corollary 1]{Xu1988jour}. It is claimed there that the proof can be obtained from \cite[The proof of Theorem 4]{Xu1988jour}, but this is not very clear since \cite[The proof of Theorem 4]{Xu1988jour} is based on the convexity of $B$ (denoted by $G$ there), which is not assumed in \cite[Corollary 1]{Xu1988jour}.  

\item\label{item:LocallyCompact} A more general version of Part \beqref{item:BoundedlyCompactProximinal} is claimed in \cite[Corollary 1]{Xu1988jour}, again without a proof: that $A$ is locally compact instead of being boundedly compact (the rest of the assumptions are the same as in Part \beqref{item:BoundedlyCompactProximinal}). It is not clear to us whether this statement is correct, and it might be that the author of \cite{Xu1988jour} actually meant ``boundedly compact'' instead of ``locally compact''.  The problem with the proof of \cite[Corollary 1]{Xu1988jour} is also mentioned in \cite[p. 128]{Narang1991jour}. 

\item\label{item:Page345} It is claimed in \cite[the remark on page 345, after the proof of (8)]{Kothe1969book} that if $X$ is a reflexive Banach space, $A$ is closed and bounded, and $B$ is closed and convex, then there is a BAP relative to $(A,B)$. This assertion is repeated in \cite[p. 128]{Narang1991jour} in an attempt to claim that \cite[Corollary 2]{Xu1988jour} (which we proved in Part \beqref{item:XuCor2}), follows from a known and more general result. However, \cite[the remark on page 345, after the proof of (8)]{Kothe1969book} is incorrect, and a simple counterexample is $X=\ell_2$, $A=\{((k+1)/k)e_k\,|\, k\in\N\}$ and $B=\{0\}$, where $e_k$ is the $k$-th basis element, namely the $k$-th component of $e_k$ is 1 and the other components are 0. Indeed, $X$ is a Hilbert space and hence reflexive, $A$ is discrete (since the distance between any two distinct elements in $A$ is at least 2) and hence closed, and $B$ is obviously closed and convex; however, $d(A,B)=1<(k+1)/k=d(e_k,B)$ for all $k\in\N$.

\item\label{itemrem:HilbertPolytop} Part \beqref{item:HilbertPolytop} is claimed without a proof in \cite[Fact 5.1(ii)]{BauschkeBorwein1994jour}. It is said there that a proof will appear in a certain future work, but eventually that specific work neither presented the claim nor presented the proof. Moreover, as far as we know, the proof of Part \beqref{item:HilbertPolytop} has not been published elsewhere (beyond the one given here).

\item\label{item:Dual} One might wonder about possible extensions of Theorem \bref{thm:Existence}. This is definitely possible. For example, see \cite[Ex. 2.16, p. 52]{Brezis2011book} (as well as Part \beqref{item:A-B_closed}) regarding the case where $X$ is a reflexive Banach space, $A$ and $B$ are closed affine subspaces of $X$ with linear parts $\wt{A}$ and $\wt{B}$, respectively, such that $\wt{A}\cap \wt{B}\neq \emptyset$ and there is some $\alpha>0$ such that $d(x,\wt{A}\cap\wt{B})\leq \alpha d(x,\wt{B})$ for each $x\in \wt{A}$. As another example, see \cite[Lemma 15D, p. 104]{Holmes1975book}, combined with Part \beqref{item:A-B_closed}, for the case where $X$ is a reflexive Banach space, $A$ and $B$ are closed and convex, $B$ is locally compact, and the intersection of the recession cones of $A$ and $B$ is $\{0\}$ (see also \cite[Facts 5.1(iv)]{BauschkeBorwein1994jour} for the real Hilbert space case, with the modification that the intersection of the recession cones should be $\{0\}$). For additional extensions, see \cite[Theorem 7]{Lin1966jour} (normed spaces) and \cite{DamaBajracharya2018jour,Narang1976jour,Narang1983jour} (metric spaces).

\item Some of the assertions formulated in Theorem \bref{thm:Existence} hold, with essentially the same proofs, in metric spaces: these are Parts \beqref{item:Nonempty}, \beqref{item:CompactProximinal} (see also \cite[Theorem 2]{Narang1976jour}), \beqref{item:BoundedlyCompactProximinal}, \beqref{item:BoundedlyCompactBounded} and \beqref{item:compact}.

\item Given a collection $(P_k)_{k\in K}$ of nonempty subsets of the ambient space $X$, called sites or generators, the Voronoi cell $V_k$ of $P_k$ is the set $V_k:=\{x\in X\,|\, d(x,P_k)\leq d(x,A_k)\}$, where $A_k:=\cup_{j\in K\backslash\{k\}}P_j$. The collection $(V_k)_{k\in K}$ is the so-called Voronoi diagram of the given sites. Voronoi diagrams  have numerous applications in science and technology: see, for example, \cite{Aurenhammer1991jour,CSKM2013,ConwaySloane,VoronoiCVD_Review,GruberLek,
OkabeBootsSugiharaChiu2000book, ReemISVD2009proc,Reem2018jour, Reem2023jour} and the references therein, and in particular see \cite{LinManochaKim2018inbook,EhmannLin2000inproc} for applications regarding the BAP problem itself. If all the sites are closed and either $K$ is finite or $K$ is infinite and the gap between the sites is positive (namely $0<\inf\{\dist(P_i,P_j)\,|\, i,j\in K, i\neq j\}$, as happens, e.g., if each site is a lattice point or a subset  located in a small neighborhood of a lattice point), then $A_k$ is closed for all $k\in K$ (in the first case this is obvious, and in the second case this follows from the fact that any sequence in $A_k$, which converges to some point in $X$, must belong to the same site $P_j$ starting from some place because of the positive gap between the sites). Therefore, if, in addition, $X$ is finite dimensional and all the sites are bounded, then Theorem \bref{thm:Existence}\beqref{item:VoronoiFiniteDim} ensures that for every $k\in K$ the distance between the Voronoi cell of $P_k$ and the union $A_k$ of the other sites is attained. 
\end{enumerate}
\end{remark}

\section{Illustrative examples and counterexamples}\label{sec:Examples}
This section presents several examples and counterexamples which illustrate some of the results established earlier.

\begin{example}\label{ex:RectangleEllipseEuclidean}
Let $(X,\|\cdot\|)$ be the Euclidean plane, $A:=[-2,2]\times [-2,0]$ be a rectangle, and $B:=\{(x_1,x_2)\in X\,|\, \frac{x_1^2}{4}+(x_2-2)^2\leq 1\}$ be an ellipse. See Figure \bref{fig:RectangleEllipseEuclidean}. Then $(X,\|\cdot\|)$ is strictly convex, both $A$ and $B$ are nonempty, convex and compact, and $B$ is actually strictly convex, and so, according to Corollary \bref{cor:StrictlyConvexNorm}, there is a unique BAP $(a_0,b_0)$ with respect to $(A,B)$. In fact, $a_0=(0,0)$ and $b_0=(0,1)$. 
\end{example}

\begin{example}\label{ex:RectangleEllipseLinfty}
Let $X:=\R^2$ be the plane with the $\|\cdot\|_{\infty}$ norm $\|(x_1,x_2)\|_{\infty}:=\max\{|x_1|,|x_2|\}$, $(x_1,x_2)\in X$, and let $A$ and $B$ be defined as in Example \bref{ex:RectangleEllipseEuclidean}. See Figure \bref{fig:RectangleEllipseLinfty}. Since $(X,\|\cdot\|_{\infty})$ is not strictly convex, the existence of a unique BAP with respect to $(A,B)$ is not guaranteed. Indeed, now $([a_0,a_1],[b_0,b_1])$ is a nondegenerate (but not a strictly nondegenerate) BAP of intervals with respect to $(\partial A,\partial B)$, where $a_0:=(-1,0)$, $a_1:=(1,0)$, and $b_0:=(0,1)=:b_1$, because $\|a(t)-b_0\|_{\infty}=1=\dist(A,B)$ for all $t\in [0,1]$.  
\end{example}

\begin{figure}[t]
\begin{minipage}[t]{0.46\textwidth}
\begin{center}{\includegraphics[trim=170 660 300 20, clip=true, scale=0.6]{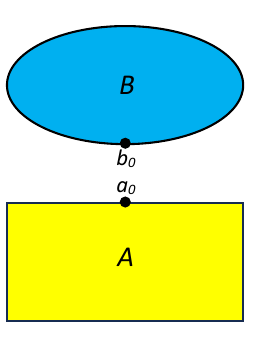}}
\end{center}
 \caption{An ellipse and a rectangle in the Euclidean plane (Example \bref{ex:RectangleEllipseEuclidean}): a unique BAP.}
\label{fig:RectangleEllipseEuclidean}
\end{minipage}
\hfill
\begin{minipage}[t]{0.49\textwidth}
\begin{center}{\includegraphics[trim=170 595 300 80, clip=true, scale=0.6]{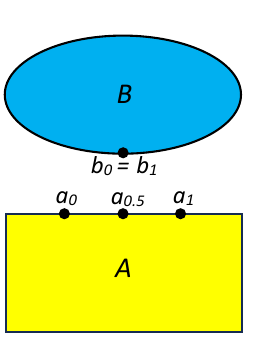}}
\end{center}
 \caption{An ellipse and a rectangle in the plane with the $\|\cdot\|_{\infty}$ norm (Example \bref{ex:RectangleEllipseLinfty}): many BAPs.}
\label{fig:RectangleEllipseLinfty}
\end{minipage}
\end{figure}

\begin{example}\label{ex:NonParallelLinfty}
Let $X:=\R^3$ with the $\|\cdot\|_{\infty}$ norm $\|(x_1,x_2,x_3)\|_{\infty}:=\max\{|x_1|,|x_2|,|x_3|\}$, $(x_1,x_2,x_3)\in X$. Let $A:=\{(x_1,x_2,x_3)\in X\,|\, x_1\in [-1,1], x_2=0, x_3=0\}$ and $B:=\{(x_1,x_2,x_3)\in X\,|\, x_1=0, x_2\in [-1,1],  x_3=h\}$ for some fixed $h\geq 1$. Then $A$ and $B$ are nondegenerate intervals. Since any $x=(x_1,x_2,x_3)\in A$ and $y=(y_1,y_2,y_3)\in B$ satisfy $|x_1-y_1|=|x_1|\leq 1$, $|x_2-y_2|=|y_2|\leq 1$ and $|x_3-y_3|=h\geq 1$, we have $\|x-y\|_{\infty}=h$ and $\dist(A,B)=h$, namely $(x,y)$ is a BAP relative to $(A,B)$ for all $(x,y)\in A\times B$. Moreover, $(A,B)$ is a strictly nondegenerate BAP of intervals with respect to $(A,B)$ although $A$ and $B$ are not parallel.   
\end{example}

\begin{example}\label{ex:EllipseCylinder}
Let $X:=\R^3$ with the Euclidean norm. Fix $\sigma_1,\sigma_2,h_1,h_2\in (0,\infty)$  and let $A$ be the elliptical cylinder defined by $A:=\{(x_1,x_2,x_3)\in X\,|\, (x_1^2/\sigma_1^2)+(x_2^2/\sigma_2^2)\leq 1, x_3\in [-h_1,0]\}$. Let $B$ be the ellipse defined by 
$B:=\{(x_1,x_2,x_3)\in X\,|\, (x_1^2/\sigma_1^2)+((x_2-\sigma_2)^2/\sigma_2^2)\leq 1, x_3=h_2\}$, namely $B$ is a translated copy of the ellipse which generates $A$. See Figure \bref{fig:EllipseCylinder}. Here $(X,\|\cdot\|)$ is strictly convex but both $A$ and $B$ are not (even though $B$ is strictly convex in the affine hull that it spans, namely when restricted to the plane $\{(x_1,x_2,x_3)\in X: x_3=h_2\}$) and indeed, $(S_1,\wt{S}_1)$ and $(S_2,\wt{S}_2)$ (see Figure \bref{fig:EllipseCylinder}) are strictly nondegenerate BAPs of intervals with respect to $(A,B)$. Of course, there are infinitely many other such pairs. 
\end{example}

\begin{example}\label{ex:TwoD}
Figure \bref{fig:TwoD} presents a two-dimensional example in which there exists at least one BAP with respect to $(A,B)$ because of Theorem \bref{thm:Existence}\beqref{item:compact}, and this pair is unique because of Proposition \bref{prop:UniquenessInNormedSpace} since there  does not exist a nondegenerate BAP of intervals with respect to $(\partial A,\partial B)$ (because if $(I_A,I_B)$ is a pair of intervals satisfying $I_A\subseteq \partial A$ and $I_B\subseteq \partial B$, then we can find a pair $(a',b')\in A\times B$ such that $\|a'-b'\|<\dist(I_A,I_B)$). Nevertheless, there do exist intervals contained in $\partial A$ which are parallel to intervals contained in $\partial B$.  
\end{example}

\begin{figure}[t]
\begin{minipage}[b]{0.47\textwidth}
\begin{center}{\includegraphics[trim=260 582 200 100, clip=true, scale=0.6]{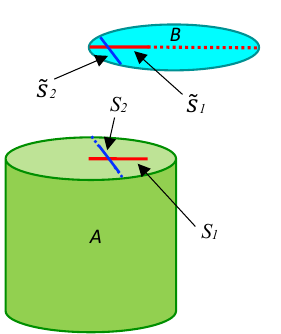}}
\end{center}
 \caption{An ellipse and a cylinder in the Euclidean space $\R^3$ (Example \bref{ex:EllipseCylinder}): two strictly nondegenerate BAPs of intervals (i.e., $(S_1,\wt{S}_1)$ and $(S_2,\wt{S}_2)$) are presented.}
\label{fig:EllipseCylinder}
\end{minipage}
\hfill
\begin{minipage}[b]{0.52\textwidth}
\begin{center}{\includegraphics[trim=110 670 200 60, clip=true, scale=0.5]{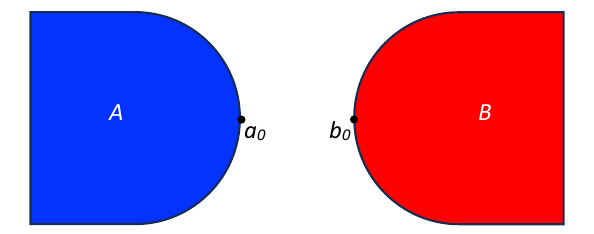}}
\end{center}
 \caption{Two shapes in the Euclidean plane whose boundaries contain parallel intervals which satisfy the conditions of Proposition \bref{prop:UniquenessInNormedSpace}, and hence induce a unique BAP (Example \bref{ex:TwoD}).}
\label{fig:TwoD}
\end{minipage}
\end{figure}


\begin{example}\label{ex:PolygonDrop}
Figure \bref{fig:PolygonDrop} presents two compact shapes in the Euclidean plane whose boundaries contain intervals, but no interval contained in the boundary of one shape is parallel to an interval contained in the boundary of the other shape. Hence the conditions of Corollary \bref{cor:StrictlyConvexNorm} and of Theorem  \bref{thm:Existence}\beqref{item:compact} are satisfied, and thus there exists a unique BAP relative to $(A,B)$. 
\end{example}

\begin{example}\label{ex:BoundedUnbounded}
Let $(X,\|\cdot\|)$ be the plane with any norm, let $A$ be a closed and bounded subset, and let $B$ be a closed and unbounded subset of the plane as in Figure \bref{fig:BoundedUnbounded}. Here the coercivity condition \beqref{eq:coercive} holds, and existence of a BAP relative to $(A,B)$ follows from Theorem \bref{thm:Existence}\beqref{item:B_D_coercive}.
\end{example}

\begin{figure}[t]
\begin{minipage}[htb]{0.46\textwidth}
\begin{center}{\includegraphics[trim=120 610 310 100, clip=true, scale=0.7]{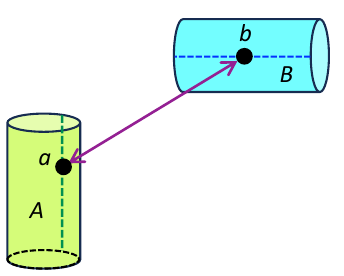}}
\end{center}
 \caption{Two non-parallel straight cylinders in the Euclidean $\R^3$ (infinite cylinders, shown partly): a unique BAP (Example \bref{ex:StraightCylinders}).}
\label{fig:ParallelStraightCylinders}
\end{minipage}
\hfill
\begin{minipage}[htb]{0.46\textwidth}
\begin{center}{\includegraphics[trim=220 675 110 85, clip=true, scale=0.7]{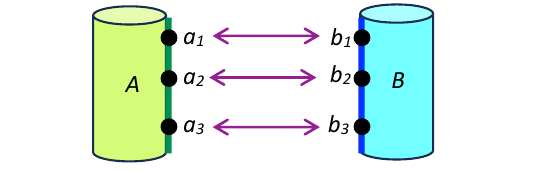}}
\end{center}
 \caption{Two parallels straight cylinders in the Euclidean $\R^3$ (infinite cylinders, shown partly): infinitely many BAPs and nondegenerate BAPs of intervals (Example \bref{ex:StraightCylinders}).}
\label{fig:NonParallelStraightCylinders}
\end{minipage}
\end{figure}

\begin{example}\label{ex:StraightCylinders}
Let $(X,\|\cdot\|)$ be the $\R^3$ with the Euclidean norm, and let $A$ and $B$ be two straight cylinders, either non-parallel (Figure \bref{fig:ParallelStraightCylinders}) or parallel (Figure \bref{fig:NonParallelStraightCylinders}). Then Theorem  \bref{thm:Existence}\beqref{item:Hypercylinders} ensures that there exists at least one BAP with respect to $(A,B)$ (in the case of non-parallel cylinders the coercivity condition \beqref{eq:coercive} holds, and existence of a BAP relative to $(A,B)$ also follows from Theorem \bref{thm:Existence}\beqref{item:B_D_coercive}). Since the norm is strictly convex,  Corollary \bref{cor:StrictlyConvexNorm} ensures the uniqueness of the BAP in the case of non-parallel cylinders.   
\end{example}

\begin{example}\label{ex:GeneralizedCylinders}
Consider $\R^3$ with any norm. Let $A$ and $B$ be two disjoint and generalized (hyper)cylinders as in Figure \bref{fig:GeneralizedCylindersNonParallel}, with axes $L_A$ and $L_B$ and bases $C_A$ and $C_B$, respectively. Then Theorem  \bref{thm:Existence}\beqref{item:Hypercylinders} ensures that there exists at least one BAP with respect to $(A,B)$: see Figure \bref{fig:GeneralizedCylindersNonParallel}. 
\end{example}

\begin{figure}[t]
\begin{minipage}[htb]{0.48\textwidth}
\begin{center}{\includegraphics[trim=140 580 70 85, clip=true, scale=0.5]{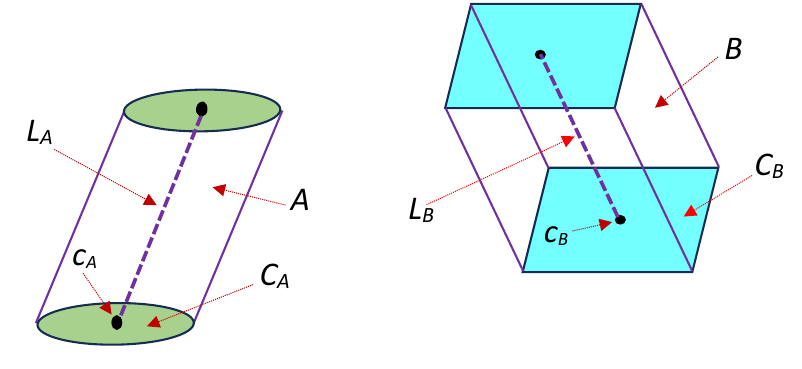}}
\end{center}
 \caption{Two non-parallel generalized (hyper)cylinders $A$ and $B$ in $\R^3$ with any norm (infinite cylinders, shown partly). At least one BAP (Example \bref{ex:GeneralizedCylinders}). }
\label{fig:GeneralizedCylindersNonParallel}
\end{minipage}
\hfill
\begin{minipage}[htb]{0.48\textwidth}
\begin{center}{\includegraphics[trim=70 560 130 80, clip=true, scale=0.4]{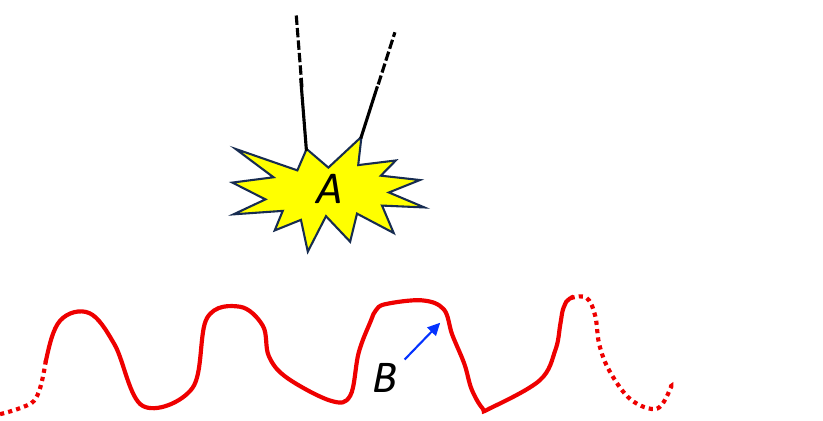}}
\end{center}
 \caption{Two closed sets in the plane (any norm), both are unbounded and still the coercivity condition \beqref{eq:coercive} holds. At least one BAP (Example \bref{ex:BoundedUnbounded}).}
\label{fig:BoundedUnbounded}
\end{minipage}
\end{figure}

\begin{figure}[t]
\begin{minipage}[b]{0.55\textwidth}
\begin{center}{\includegraphics[trim=130 632 240 90, clip=true, scale=0.57]{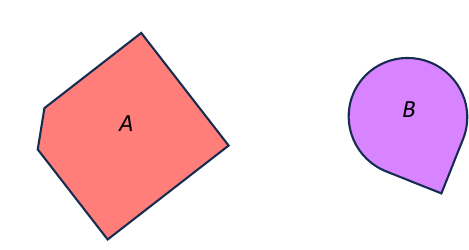}}
\end{center}
 \caption{Two shapes in the Euclidean plane for which the conditions of Corollary \bref{cor:StrictlyConvexNorm}  are satisfied (Example \bref{ex:PolygonDrop}). A unique BAP.}
\label{fig:PolygonDrop}
\end{minipage}
\end{figure}

\begin{example}\label{ex:IntersectionStrictlyConvexSets}
An illustration of Theorem \bref{thm:IntersectionStrictlyConvexSets}
in the plane, with any norm, is presented in Figure \bref{fig:IntersectionStrictlyConvexSets}. This theorem and Theorem \bref{thm:Existence}\beqref{item:compact} imply that there is a unique BAP relative to $(A,B)$. 
\end{example}

\begin{example}\label{ex:IntersectionStrictlyConvexNorm}
An illustration of Theorem \bref{thm:IntersectionStrictlyConvexNorm} in the plane, with any strictly convex norm, is presented in Figure \bref{fig:IntersectionStrictlyConvexNorm}. This theorem, as well as Theorem \bref{thm:Existence}\beqref{item:compact}, ensures that there is a unique BAP relative to $(A,B)$. 
\end{example}


\begin{figure}[t]
\begin{minipage}[b]{0.45\textwidth}
\begin{center}{\includegraphics[trim=200 602 70 100, clip=true, scale=0.583]{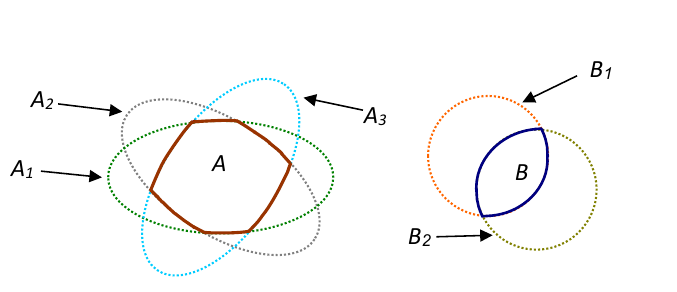}}
\end{center}
 \caption{An illustration of Theorem \bref{thm:IntersectionStrictlyConvexSets}: two intersections (solid lines) of strictly convex subsets (dashed lines) in the  plane with any norm (Example \bref{ex:IntersectionStrictlyConvexSets}). Only the boundaries of the subsets are shown.}
\label{fig:IntersectionStrictlyConvexSets}
\end{minipage}
\begin{minipage}[b]{0.5\textwidth}
\begin{center}{\includegraphics[trim=100 600 100 110, clip=true, scale=0.58]{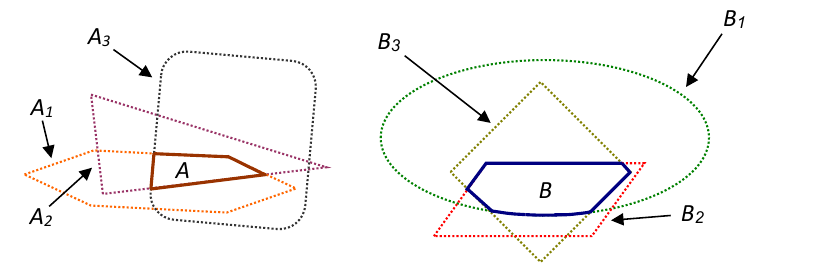}}
\end{center}
 \caption{Two intersections  (solid lines) of subsets (dashed lines) which satisfy the conditions of Theorem \bref{thm:IntersectionStrictlyConvexNorm} and are  located in the plane with any strictly convex norm (Example \bref{ex:IntersectionStrictlyConvexNorm}). Only the boundaries of the sets are shown.}
\label{fig:IntersectionStrictlyConvexNorm}
\end{minipage}
\end{figure}


\section*{Appendix}
This section presents the proofs the claims which were formulated in Section \bref{sec:Auxiliary}.

\begin{proof}[Proof of Lemma \bref{lem:DistClosures}]
Since $A\times B\subseteq \ol{A}\times\ol{B}$, we have $\dist(\ol{A},\ol{B})=\inf\{d(a,b)\,|\, (a,b)\in \ol{A}\times\ol{B}\}\leq \inf\{d(a,b)\,|\, (a,b)\in A\times B\}=\dist(A,B)$. For the converse direction, let $\epsilon>0$ be arbitrary. By the definitions of $\dist(\ol{A},\ol{B})$, $\ol{A}$ and $\ol{B}$ there are $\wt{a}\in\ol{A}$ and $\wt{b}\in \ol{B}$ such that $d(\wt{a},\wt{b})<\dist(\ol{A},\ol{B})+0.5\epsilon$ and there are some $a\in A$ and $b\in B$ such that $d(\wt{a},a)<0.25\epsilon$ and $d(\wt{b},b)<0.25\epsilon$. Thus, by the triangle inequality and because $\dist(A,B)\leq d(a,b)$, we have $\dist(A,B)\leq d(a,b)\leq d(a,\wt{a})+d(\wt{a},\wt{b})+d(\wt{b},b)<0.25\epsilon+\dist(\ol{A},\ol{B})+0.5\epsilon+0.25\epsilon=\dist(\ol{A},\ol{B})+\epsilon$. This inequality is true for all $\epsilon>0$ and hence $\dist(A,B)\leq \dist(\ol{A},\ol{B})$. The final assertion is clear because if $(a_0,b_0)\in A\times B$ and $a_0=b_0$, then $A\cap B\neq\emptyset$, a contradiction. Thus $a_0\neq b_0$ and $\dist(A,B)=d(a_0,b_0)>0$, as required. 
\end{proof}

\begin{proof}[Proof of Lemma \bref{lem:DistBoundary}]
Let $\epsilon>0$ be arbitrary. By the definition of $\dist(A,B)$ there are $a\in A$ and $b\in B$ which satisfy $\|a-b\|<\dist(A,B)+\epsilon$. Let $g:[0,1]\to X$ be defined by $g(t):=a+t(b-a)$ for every $t\in [0,1]$. Since $g(0)=a\in A$ and $g(1)=b\in B$, the subsets $\{t\in [0,1]\,|\, g(t)\in A\}$ and $\{t\in [0,1]\,|\, g(t)\in B\}$ are  nonempty. Let $t_1:=\sup\{t\in [0,1]\,|\, g(t)\in A\}$. Since $g$ is continuous, for all $\eta>0$ there is $\delta>0$ such that $\|g(t)-g(t_1)\|<\eta$ whenever $|t-t_1|<\delta$ and $t\in [0,1]$. Since $t_1-\delta<t_1$, the definition of $t_1$ implies that there is $s\in \{t\in [0,1]\,|\, g(t)\in A\}$ such that $t_1-\delta<s\leq t_1$. This $s$ satisfies $\|g(s)-g(t_1)\|<\eta$ by the choice of $\delta$. Since $\eta$ was an arbitrary positive number and since $g(s)\in A$, it follows that in any neighborhood of $g(t_1)$ there are points from $A$. Now, if $t_1=1$, then $g(t_1)=b\notin A$ because $A\cap B=\emptyset$, and hence obviously in any neighborhood of $g(t_1)$ there are points outside $A$. Thus $g(t_1)\in\partial A$. Otherwise $t_1<1$, but then $g(t)\notin A$ whenever $t_1<t<1$ by the definition of $t_1$, and since $g(t_1)=\lim_{t\to t_1,t>t_1}g(t)$ by  the continuity of $g$, we see that also in this case in any neighborhood of $g(t_1)$ there are points outside $A$. Hence $g(t_1)\in\partial A$. Similarly, $t_2:=\inf\{t\in [0,1]\,|\, g(t)\in B\}$ satisfies $g(t_2)\in \partial B$. In particular $\partial A\neq\emptyset$ and $\partial B\neq\emptyset$. Thus $\dist(\partial A,\partial B)\leq \|g(t_1)-g(t_2)\|=|t_1-t_2|\|a-b\|\leq \|a-b\|< \dist(A,B)+\epsilon$. Since $\epsilon$ was an  arbitrary positive number, we conclude that $\dist(\partial A,\partial B)\leq \dist(A,B)$. Conversely, since $\partial A\subseteq \ol{A}$ and $\partial B\subseteq \ol{B}$, we have $\dist(\ol{A},\ol{B})\leq \dist(\partial A,\partial B)$, and therefore Lemma \bref{lem:DistClosures} implies that  $\dist(A,B)=\dist(\ol{A},\ol{B})\leq \dist(\partial A,\partial B)$. 
\end{proof}

\begin{proof}[Proof of Lemma \bref{lem:a(t)b(t)}]
Condition \beqref{item:[0,1]} obviously implies Condition \beqref{item:0,1}. Conversely, assume that Condition \beqref{item:0,1} holds. Since both $(a_0,b_0)$ and $(a_1,b_1)$ are BAPs with respect to $(A,B)$, we have $\|a_0-b_0\|=\dist(A,B)=\|a_1-b_1\|$. By the convexity of $A$ and $B$ and the fact that $a(t)=(1-t)a_0+ta_1$ and $b(t)=(1-t)b_0+tb_1$, we have $a(t)\in A$ and $b(t)\in B$ for each $t\in [0,1]$. Thus $\dist(A,B)\leq\|a(t)-b(t)\|$. On the other hand, the triangle inequality and Condition \beqref{item:0,1} imply that $\|a(t)-b(t)\|\leq (1-t)\|a_0-b_0\|+t\|a_1-b_1\|=(1-t)\dist(A,B)+t\dist(A,B)=\dist(A,B)$ for all $t\in [0,1]$. Hence $\|a(t)-b(t)\|=\dist(A,B)$ for all $t\in [0,1]$, namely Condition \beqref{item:[0,1]} holds.
\end{proof}

\begin{proof}[Proof of Lemma \bref{lem:a_0b_0Boundary}]
Suppose first that $(a_0,b_0)$ is a BAP with respect to $(A,B)$. Then $\|a_0-b_0\|=\dist(A,B)$, and because of Lemma \bref{lem:DistBoundary}, we have $\dist(A,B)=\dist(\partial A,\partial B)$, namely,  $\|a_0-b_0\|=\dist(\partial A,\partial B)$. Therefore, in order to prove that  $(a_0,b_0)$ is a BAP with respect to $(\partial A,\partial B)$ it remains to show that $(a_0,b_0)\in \partial A\times\partial B$. Let $r>0$ be arbitrary and consider the open ball  with radius $r$ around $a_0$. Obviously $a_0\in A$ is in the ball. In addition, let $t\in (0,\min\{r/(\|b_0-a_0\|+1),1\})$. A simple calculation shows that $a(t):=a_0+t(b_0-a_0)$ is both in  $(a_0,b_0]$ and in the ball, and since $\|a(t)-b_0\|=(1-t)\|a_0-b_0\|<\|a_0-b_0\|=\dist(A,B)$, it follows that $a(t)\notin A$  because the distance between a point from $A$ and a point from $B$ is at least $\dist(A,B)$. Since $r>0$ can be arbitrarily small, we have $a_0\in\partial A$. Similarly, $b_0\in\partial B$. Conversely, if $A$ and $B$ are closed and that $(a_0,b_0)$ is a BAP with respect to $(\partial A,\partial B)$, then $a_0\in\partial A\subseteq A$,  $b_0\in\partial B\subseteq B$ and $\|a_0-b_0\|=\dist(\partial A,\partial B)=\dist(A,B)$, where the last equality is by Lemma \bref{lem:DistBoundary}. Hence $(a_0,b_0)$ is a BAP with respect to $(A,B)$.
\end{proof}

\begin{proof}[Proof of Lemma \bref{lem:BestApproxIntervalBoundary}]
Let $t\in [0,1]$ be arbitrary. Lemma \bref{lem:a(t)b(t)}  implies that $(a(t),b(t))$ is a BAP with respect to $(A,B)$, and so $\|a(t)-b(t)\|=\dist(A,B)$. Because $(a(t),b(t))$ is a BAP with respect to $(A,B)$, Lemma \bref{lem:a_0b_0Boundary} (with $(a(t),b(t))$ instead of $(a_0,b_0)$) implies that $(a(t),b(t))$ is also a BAP with respect to $(\partial A,\partial B)$. Thus, $a(t)\in\partial A$, $b(t)\in\partial B$ and $\|a(t)-b(t)\|=\dist(\partial A,\partial B)$. Since $[a_0,a_1]=\{a(t)\,|\, t\in [0,1]\}$ and $[b_0,b_1]=\{b(t)\,|\, t\in [0,1]\}$, we have $[a_0,a_1]\subseteq \partial A$ and $[b_0,b_1]\subseteq \partial B$. Thus, Definition \bref{def:BestApproxPairIntervals} implies that $([a_0,a_1],[b_0,b_1])$ is a BAP of intervals with respect to both $(A,B)$ and $(\partial A, \partial B)$. 
\end{proof}

\begin{proof}[Proof of Lemma \bref{lem:StrictlyNondegenerate}]
Since $([a_0,a_1],[b_0,b_1])$ is a BAP of intervals with respect to $(A,B)$, we have, in particular, that $(a_0,b_0)$ and $(a_1,b_1)$ are BAPs with respect to $(A,B)$. Hence Lemma \bref{lem:BestApproxIntervalBoundary} implies that  $([a_0,a_1],[b_0,b_1])$ is a BAP of intervals with respect to $(\partial A, \partial B)$. In particular, $[a_0,a_1]\subseteq \partial A$ and $[b_0,b_1]\subseteq \partial B$.

It remains to be shown that the pair $([a_0,a_1],[b_0,b_1])$ is strictly nondegenerate. Since this pair is a nondegenerate BAP of intervals with respect to $(A,B)$, we have that either $a_0\neq a_1$ or $b_0\neq b_1$. Assume that $a_0\neq a_1$: the case $b_0\neq b_1$ can be treated similarly. If, for the sake of contradiction, $b_0=b_1$, then this  equality and the fact that  $(a_0,b_0)$ and $(a_1,b_1)$ are BAPs with respect to $(A,B)$ imply that $\|a_0-b_0\|=\dist(A,B)=\|a_1-b_1\|=\|a_1-b_0\|$. This equality means that both $a_0$ and $a_1$ are located on the boundary of the ball whose center is $b_0$ and its radius is $\dist(A,B)$, which is positive according to Lemma \bref{lem:DistClosures}. Since $a_0\neq a_1$ and the space is strictly convex, the open interval $(a_0,a_1)$ is strictly inside this ball, and so, in particular, $a_{0.5}$ is strictly inside this ball. Thus, $\|a_{0.5}-b_0\|<\dist(A,B)$. On the other hand, since $a_{0.5}\in A$ by the convexity of $A$ and since $b_0\in B$, the definition of $\dist(A,B)$ implies that $\|a_{0.5}-b_0\|\geq \dist(A,B)$, a contradiction. Hence $b_0\neq b_1$ and indeed $([a_0,a_1],[b_0,b_1])$ is a strictly nondegenerate BAP of intervals relative to $(A,B)$ (and relative to $(\partial A,\partial B)$). 
\end{proof}

\begin{proof}[Proof of Lemma \bref{lem:parallel}]
Since $(a_0,b_0)$ and $(a_1,b_1)$ are BAPs with respect to $(A,B)$, Lemma \bref{lem:BestApproxIntervalBoundary} implies that $([a_0,a_1],[b_0,b_1])$ is a BAP of intervals with respect to $(A,B)$, and since $(a_0,b_0)\neq (a_1,b_1)$, either $[a_0,a_1]$ is nondegenerate or $[b_0,b_1]$ is nondegenerate. Thus $([a_0,a_1],[b_0,b_1])$ is a nondegenerate BAP of intervals with respect to $(A,B)$, and since $X$ is strictly convex Lemma \bref{lem:StrictlyNondegenerate} implies that both  $[a_0,a_1]$ and  $[b_0,b_1]$ are nondegenerate. In addition, Lemma \bref{lem:BestApproxIntervalBoundary} implies that $[a_0,a_1]\subseteq \partial A$ and $[b_0,b_1]\subseteq \partial B$. 
 
It remains to be shown that $[a_0,a_1]$ and  $[b_0,b_1]$ are strictly parallel. Lemma \bref{lem:a(t)b(t)}  implies that $(a(t),b(t))$ is a BAP with respect to $(A,B)$ for each $t\in [0,1]$. Thus,  if we define $g:[0,1]\to [0,\infty)$ by
\begin{equation}\label{eq:f}
g(t):=\|a(t)-b(t)\|=\|a_0-b_0+(a_1-b_1-(a_0-b_0))t\|, \quad \forall t\in [0,1],
\end{equation} 
then we have $g(t)=\dist(A,B)$ for all $t\in [0,1]$. In particular, 
\begin{equation}\label{eq:f(0.5)=dist(A,B)}
g(0.5)=\dist(A,B).
\end{equation}
Assume for the sake of contradiction that $a_0-b_0\neq a_1-b_1$. Since $\|a_0-b_0\|=g(0)=g(1)=\|a_1-b_1\|=\dist(A,B)$ and since $\dist(A,B)>0$ according to Lemma \bref{lem:DistClosures}, the distinct points $a_0-b_0$ and $a_1-b_1$ are located on the boundary of the ball of positive radius $\dist(A,B)$ around the origin. Since $(X,\|\cdot\|)$ is strictly convex,  $\|0.5(a_0-b_0)+0.5(a_1-b_1)\|<\dist(A,B)$. But from \beqref{eq:f} we have $g(0.5)=\|0.5(a_0-b_0)+0.5(a_1-b_1)\|$. Therefore $g(0.5)<\dist(A,B)$, a contradiction to \beqref{eq:f(0.5)=dist(A,B)}. Thus $a_0-b_0=a_1-b_1$ and hence $v:=a_1-a_0=b_1-b_0$. Since we already know that $b_0\neq b_1$ (because $[b_0,b_1]$ is nondegenerate as we showed earlier), we also have $v\neq 0$.

Consider the real lines $L:=\{a_0+tv: t\in \R\}$ and $M:=\{b_0+sv: s\in \R\}$. By letting $t,s\in [0,1]$ in the definitions of $L$ and $M$, we see that $[a_0,a_1]\subset L$ and $[b_0,b_1]\subset M$. We claim that $b_0\notin L$. Indeed, suppose for the sake of contradiction that $b_0\in L$. Then $b_0=a_0+tv$ for some $t\in\R$ and hence either $t\in [0,1]$ or $t>1$ or $t\in [-1,0)$ or $t<-1$. If $t\in [0,1]$, then $b_0\in [a_0,a_1]\subseteq A$, a contradiction since we assume that $A\cap B=\emptyset$. If $t>1$, then $\|b_0-a_1\|=\|(a_0+tv)-(a_0+v)\|=(t-1)\|v\|<t\|v\|=\|b_0-a_0\|=\dist(A,B)$, a contradiction to the minimality of $\dist(A,B)$. If $t\in [-1,0)$, then $t+1\in [0,1)$, and so $b_1=b_0+v=a_0+(t+1)v\in [a_0,a_1]\subseteq A$, a contradiction to the assumption $A\cap B=\emptyset$. Therefore only the case $t<-1$ remains; in this case $t+1<0$, and since $b_1=b_0+v=a_0+(t+1)v$, we have $\|b_1-a_0\|=|t+1|\|v\|=-(t+1)\|v\|<-t\|v\|=\|b_0-a_0\|=\dist(A,B)$, a contradiction to the minimality of $\dist(A,B)$. As a result, indeed $b_0\notin L$. 

It must be that $v$ and $b_0-a_0$ are linearly independent over $\R$. Indeed, if, for the sake of contradiction, $\lambda_1(b_0-a_0)+\lambda_2 v=0$ for a pair of real scalars $(\lambda_1,\lambda_2)\neq (0,0)$, then $\lambda_1\neq 0$ since otherwise $\lambda_2 v=0$, and since the assumption that $(\lambda_1,\lambda_2)\neq (0,0)$ implies that $\lambda_2\neq 0$, we have $v=0$, a  contradiction; thus $\lambda_1\neq 0$, and hence $b_0=a_0+(-\lambda_2/\lambda_1)v$, that is, $b_0\in L$, a contradiction to what has been proved in the previous paragraph. In addition, $L\cap M=\emptyset$, since otherwise $a_0+tv=b_0+sv$ for some $t,s\in\R$, and so $b_0=a_0+(t-s)v\in L$, a contradiction. Thus $L$ and $M$ are strictly parallel since their intersection is the empty set and both of them are located on the same real two-dimensional affine subspace  (namely on $a_0+\textnormal{span}\{v,b_0-a_0\}$). Hence $[a_0,a_1]$  and $[b_0,b_1]$ are strictly parallel since they are located on the strictly parallel real lines $L$ and $M$, respectively. 
\end{proof}

\begin{proof}[Proof of Lemma \bref{lem:ThreePointsIntervalBoundary}]
Let $w\in [x,z]$. Since $x$ and $z$ are in $\overline{C}$  and since $\overline{C}$ is convex (because $C$ is convex, see, e.g., \cite[Theorem 2.23(a), p. 28]{VanTiel1984book}), we have $w\in \overline{C}$. Since $[x,z]=[x,y]\cup [y,z]
$, either $w\in [x,y]$ or $w\in [y,z]$. Suppose that the first case holds. If $w=x$ or $w=y$, then $w\in \partial C$ by our assumption on $x$ and $y$. Otherwise, $w\in (x,y)$. Assume for the sake of contradiction that $w\notin \partial C$. This assumption and the fact that $w\in \overline{C}$ imply that $w\in \Int(C)$. Since, as is well known \cite[Theorem 2.23(b), p. 28]{VanTiel1984book}, the half-open line segment  between an interior point of $C$ and a point in $\overline{C}$ is contained in $\Int(C)$, we have $[w,z)\subseteq \Int(C)$. From the fact that $y\in [w,z)$ we conclude that $y\in \Int(C)$, a contradiction to our assumption that $y\in\partial C$.  Hence $w$ must be in $\partial C$, and since $w$ is an arbitrary point in $[x,y]$, we conclude that $[x,y]\subseteq \partial C$.  Similarly, $[y,z]\subseteq \partial C$. Thus $[x,z]\subseteq \partial C$. 
\end{proof}

\begin{proof}[Proof of Lemma \bref{lem:(Un)BoundedMinimizing}]
\begin{enumerate}[(i)]
\item If both $(a_k)_{k\in\infty}$ and $(b_k)_{k\in\N}$ are bounded, then we are done. Otherwise, one of these sequences, say $(a_k)_{k\in\infty}$, is unbounded. Hence there is an infinite subset $N_1$ of $\N$ such that $\lim_{k\to\infty, k\in N_1}\|a_k\|=\infty$. Since $(a_k-b_k)_{k\in\N}$ converges to the finite number $\dist(A,B)$, this sequence is bounded. Since $b_k=a_k-(a_k-b_k)$ for all $k\in N_1$, it follows that $(b_k)_{k\in N_1}$ is a difference between an unbounded sequence and a bounded one, and therefore   $\lim_{k\to\infty, k\in N_1}\|b_k\|=\infty$. Thus $(b_k)_{k\in\N}$ is unbounded too.

\item The assertion obviously holds if $A\cup B$ is bounded since then $A$ and $B$ are bounded, and so are any sequences contained in them. Now assume that $A\cup B$ is unbounded and \beqref{eq:coercive} holds. If, say, $(a_k)_{k\in\infty}$ is unbounded, then by Part \beqref{item:BoundedOrUnbounded} there is an infinite subset $N_1$ of $\N$ such that $\lim_{k\to\infty, k\in N_1}\|a_k\|=\lim_{k\to\infty, k\in N_1}\|b_k\|=\infty$. But then \beqref{eq:coercive} implies that $\lim_{k\to\infty, k\in N_1}\|a_k-b_k\|=\infty$, a contradiction to \beqref{x_ky_xdist(A,B)}. Thus $(a_k)_{k\in\infty}$ is bounded, and from Part \beqref{item:BoundedOrUnbounded} also $(b_k)_{k\in\infty}$ is bounded.

\item Suppose that $(a,b)=(w)\lim_{k\to\infty, k\in N_1}(a_k,b_k)$ for some infinite subset $N_1$ of $\N$. Then $a=(w)\lim_{k\to\infty, k\in N_1}a_k$ and $b=(w)\lim_{k\to\infty, k\in N_1}b_k$, and hence   $a-b=(w)\lim_{k\to\infty, k\in N_1}(a_k-b_k)$. Since the norm is weakly sequentially lower semicontinuous  \cite[II.3.27, p. 68]{DunfordSchwartz1958book}, we have $\|a-b\|\leq \lim_{k\to\infty, k\in N_1}\|a_k-b_k\|=\dist(A,B)$. On the other hand $\dist(A,B)\leq \|a-b\|$ since $(a,b)\in A\times B$. Hence $\|a-b\|=\dist(A,B)$ and $(a,b)$ is a BAP relative to $(A,B)$. 
\end{enumerate}
\end{proof}

\begin{proof}[Proof of Lemma \bref{lem:WeaklyBolzanoWeierstrass}]
Let $(c_k)_{k\in\N}$ be an arbitrary bounded sequence in $C$. We need to show that there exists an infinite subset $K\subseteq \N$ and a point $c\in C$ such that $(w)\lim_{k\to\infty, k\in K}c_k=c$. Fix an arbitrary point $z\in C$. If $\epsilon:=\liminf\{\|c_k-z\|\,|\,k\in\N\}=0$, then the definition of the liminf implies that there is an infinite subset $K\subseteq \N$ such that $\lim_{k\to\infty, k\in K}\|c_k-z\|=0$. Since $(c_k)_{k\in K}$ converges strongly to $z$, it also converges weakly to $z$, and so we are done (with $c:=z$). Otherwise $\epsilon>0$, and so $\|c_k-z\|\geq\epsilon>0$ for all but finitely many $k\in\N$, say for all $k\in \N$ which is greater than some $k_0\in\N$. 

Since $C$ is closed and norm-locally weakly sequentially compact, there is a closed ball $D$, centered at $z$, with radius $r\in (0,\epsilon)$, whose intersection with $C$ is weakly sequentially compact. Since $(c_k)_{k\in\N}$ is bounded, there is some $\rho>0$ such that $\|c_k\|<\rho$ for all $k\in\N$. Define $\alpha_k:=0.5r/\|c_k-z\|$ for all $k_0\leq k\in\N$. Then $\alpha_k\in [0.5r/(\rho+\|z\|),0.5r/\epsilon]$ for all $k_0\leq k\in\N$  by the triangle inequality and the choice of $r$ and $\rho$. Hence the compactness of the real-line interval $[0.5r/(\rho+\|z\|),0.5r/\epsilon]$ implies that there is an infinite subset $S\subseteq \{k_0,k_0+1,k_0+2,\ldots\}$ and a real number $\alpha\in [0.5r/(\rho+\|z\|),0.5r/\epsilon]$ such that $\lim_{k\to \infty, k\in S}\alpha_k=\alpha$. 

Define $c'_k:=z+\alpha_k(c_k-z)$ for all $k\in S$. Then 
$c'_k\in D$ for every $k\in S$. Moreover, $c'_k\in [z,c_k]\subseteq C$ because $C$ is convex and $\alpha_k\in [0,1]$. Hence $c'_k\in C\cap D$ for all $k\in S$, and therefore, since $C\cap D$ is weakly sequentially compact, there is a point $c'\in C\cap D$ and an infinite subset $K\subseteq S$ such that $(w)\lim_{k\to\infty, k\in K}c'_k=c'$. Since $\alpha_k\geq 0.5r/(\rho+\|z\|)>0$ for all $k\geq k_0$ and $\lim_{k\in K}\alpha_k=\alpha>0$, simple arithmetic and arithmetic of (weak) limits yields
\begin{equation*}
(w)\lim_{k\in K}c_k=(w)\lim_{k\in K}\left[z+\frac{1}{\alpha_k}\left(c_k'-z\right)\right]=z+\frac{1}{\alpha}\left(c'-z\right)=:c.
\end{equation*}
Since $C$ is closed and convex, it is also weakly closed \cite[Corollary 1.5, p. 126]{Conway1990book}, and hence, because $(c_k)_{k\in K}$ is in $C$, also $c\in C$. 
\end{proof}
 
\section*{Acknowledgments} \textmd{D.R. thanks Simeon Reich, Itai Shafrir and Constantin Z{\u a}linescu for helpful discussions related to 
\cite{BoulosReich2015jour,MatouskovaReich2003jour,BrezisMironescuShafrir2016jour,
RubinsteinShafrir2007jour,LeviShafrir2014jour,VoiseiZalinescu2011jour,Zalinescu2022jour}. This research is supported by the ISF Grant Number 2874/19 within the ISF-NSFC joint research program. The work of Y.C. was also supported by U. S. National Institutes of Health Grant Number R01CA266467 and by the Cooperation Program in Cancer Research of the German Cancer Research Center (DKFZ) and the Israeli Ministry of Innovation, Science and Technology (MOST).}


\bibliographystyle{acm}
\bibliography{biblio}

\begin{thebibliography}{100}

\bibitem{AharoniCensorJiang2018jour}
{\sc Aharoni, R., Censor, Y., and Jiang, Z.}
\newblock Finding a best approximation pair of points for two polyhedra.
\newblock {\em Comput. Optim. Appl. 71\/} (2018), 509--523.

\bibitem{Aurenhammer1991jour}
{\sc Aurenhammer, F.}
\newblock {V}oronoi diagrams - a survey of a fundamental geometric data
  structure.
\newblock {\em ACM Comput. Surv. 3}, 3 (1991), 345--405.

\bibitem{BoundaryIntersection2016misc}
{\sc Bancerek, G., and Prime.mover}.
\newblock Boundary of intersection is subset of union of boundaries.
\newblock {\em Proof Wiki\/} (2016).
\newblock (Updated: 2021, retrieved: July 18, 2023).

\bibitem{BauschkeBorwein1993jour}
{\sc Bauschke, H.~H., and Borwein, J.~M.}
\newblock On the convergence of von {N}eumann's alternating projection
  algorithm for two sets.
\newblock {\em Set-Valued Anal. 1\/} (1993), 185--212.

\bibitem{BauschkeBorwein1994jour}
{\sc Bauschke, H.~H., and Borwein, J.~M.}
\newblock Dykstra's alternating projection algorithm for two sets.
\newblock {\em J. Approx. Theory 79\/} (1994), 418--443.

\bibitem{BauschkeBorwein1996jour}
{\sc Bauschke, H.~H., and Borwein, J.~M.}
\newblock On projection algorithms for solving convex feasibility problems.
\newblock {\em SIAM Rev. 38\/} (1996), 367--426.

\bibitem{BauschkeBorweinLewis1997inproc}
{\sc Bauschke, H.~H., Borwein, J.~M., and Lewis, A.~S.}
\newblock The method of cyclic projections for closed convex sets in {H}ilbert
  space.
\newblock {\em Contemp. Math. 204\/} (1997), 1--38.

\bibitem{BauschkeCombettes2017book}
{\sc Bauschke, H.~H., and Combettes, P.~L.}
\newblock {\em Convex {A}nalysis and {M}onotone {O}perator {T}heory in
  {H}ilbert {S}paces}, 2~ed.
\newblock CMS Books in Mathematics. Springer International Publishing, Cham,
  Switzerland, 2017.

\bibitem{BauschkeCombettesLuke2004jour}
{\sc Bauschke, H.~H., Combettes, P.~L., and Luke, D.~R.}
\newblock Finding best approximation pairs relative to two closed convex sets
  in {H}ilbert spaces.
\newblock {\em J. Approx. Theory 127\/} (2004), 178--192.

\bibitem{BauschkeSinghWang2022jour}
{\sc Bauschke, H.~H., Singh, S., and Wang, X.}
\newblock Finding best approximation pairs for two intersections of closed
  convex sets.
\newblock {\em Comput. Optim. Appl. 81\/} (2022), 289--308.

\bibitem{Beauzamy1982book}
{\sc Beauzamy, B.}
\newblock {\em Introduction to {B}anach {S}paces and their {G}eometry}, vol.~68
  of {\em North-Holland Mathematics Studies}.
\newblock North-Holland Publishing Co., Amsterdam-New York, 1982.
\newblock Notas de Matem{\'a}tica [Mathematical Notes], 86.

\bibitem{BehlingBello-CruzSantos2021jour}
{\sc Behling, R., Bello-Cruz, Y., and Santos, L.-R.}
\newblock Infeasibility and error bound imply finite convergence of alternating
  projections.
\newblock {\em SIAM J. Optim. 31}, 4 (2021), 2863--2892.

\bibitem{BoulosReich2015jour}
{\sc Boulos, W., and Reich, S.}
\newblock Porosity results for two-set nearest and farthest point problems.
\newblock {\em Rend. Circ. Mat. Palermo (2) 64}, 3 (2015), 493--507.

\bibitem{Brezis2011book}
{\sc Brezis, H.}
\newblock {\em Functional {A}nalysis, {S}obolev {S}paces and {P}artial
  {D}ifferential {E}quations}.
\newblock Universitext. Springer, New York, 2011.

\bibitem{BrezisMironescuShafrir2016jour}
{\sc Brezis, H., Mironescu, P., and Shafrir, I.}
\newblock Distances between homotopy classes of {$W^{s,p}(\Bbb S^N;\Bbb S^N)$}.
\newblock {\em ESAIM Control Optim. Calc. Var. 22}, 4 (2016), 1204--1235.

\bibitem{Bruck1973jour}
{\sc Bruck, Jr., R.~E.}
\newblock Properties of fixed-point sets of nonexpansive mappings in {B}anach
  spaces.
\newblock {\em Trans. Amer. Math. Soc. 179\/} (1973), 251--262.

\bibitem{ButnariuCensorGurfilHadar2008jour}
{\sc Butnariu, D., Censor, Y., Gurfil, P., and Hadar, E.}
\newblock On the behavior of subgradient projections methods for convex
  feasibility problems in {E}uclidean spaces.
\newblock {\em SIAM J. Optim. 19\/} (2008), 786--807.

\bibitem{ByrneCensor2001jour}
{\sc Byrne, C., and Censor, Y.}
\newblock Proximity function minimization using multiple {B}regman projections,
  with applications to split feasibility and {K}ullback-{L}eibler distance
  minimization.
\newblock {\em Ann. Oper. Res. 105\/} (2001), 77--98.

\bibitem{CameronCulley1986inproc}
{\sc Cameron, S., and Culley, R.}
\newblock Determining the minimum translational distance between two convex
  polyhedra.
\newblock In {\em Proceedings of the 1986 IEEE International Conference on
  Robotics and Automation\/} (1986), vol.~3, pp.~591--596.

\bibitem{CaseiroFacasVicenteVitoria2019jour}
{\sc Caseiro, R., Facas~Vicente, M.~A., and Vit\'{o}ria, J.}
\newblock Projection method and the distance between two linear varieties.
\newblock {\em Linear Algebra Appl. 563\/} (2019), 446--460.

\bibitem{Cegielski2012book}
{\sc Cegielski, A.}
\newblock {\em Iterative {M}ethods for {F}ixed {P}oint {P}roblems in {H}ilbert
  {S}paces}, vol.~2057 of {\em Lecture Notes in Mathematics}.
\newblock Springer, Heidelberg, 2012.

\bibitem{CegielskiSuchocka2008jour}
{\sc Cegielski, A., and Suchocka, A.}
\newblock Relaxed alternating projection methods.
\newblock {\em SIAM J. Optim. 19}, 3 (2008), 1093--1106.

\bibitem{CensorCegielski2015jour}
{\sc Censor, Y., and Cegielski, A.}
\newblock Projection methods: an annotated bibliography of books and reviews.
\newblock {\em Optimization 64\/} (2015), 2343--2358.

\bibitem{CensorMansourReem2024jour}
{\sc Censor, Y., Mansour, R., and Reem, D.}
\newblock The alternating simultaneous {H}alpern-{L}ions-{W}ittmann-{B}auschke
  algorithm for finding the best approximation pair for two disjoint
  intersections of convex sets.
\newblock {\em J. Approx. Theory 301\/} (2024), Paper No. 106045, 22 pp.

\bibitem{CensorReem2015jour}
{\sc Censor, Y., and Reem, D.}
\newblock Zero-convex functions, perturbation resilience, and subgradient
  projections for feasibility-seeking methods.
\newblock {\em Math. Prog. (Ser. A) 152\/} (2015), 339--380.

\bibitem{CensorZaknoon2018jour}
{\sc Censor, Y., and Zaknoon, M.}
\newblock Algorithms and convergence results of projection methods for
  inconsistent feasibility problems: a review.
\newblock {\em Pure Appl. Funct. Anal. 3\/} (2018), 565--586.

\bibitem{CensorZenios1997book}
{\sc Censor, Y., and Zenios, A.~S.}
\newblock {\em Parallel {O}ptimization: {T}heory, {A}lgorithms, and
  {A}pplications}.
\newblock Numerical Mathematics and Scientific Computation. Oxford University
  Press, New York, 1997.
\newblock With a foreword by George B. Dantzig.

\bibitem{ChangChoiKimWang2011jour}
{\sc Chang, J.-W., Choi, Y.-K., Kim, M.-S., and Wang, W.}
\newblock Computation of the minimum distance between two {B}\'ezier
  curves/surfaces.
\newblock {\em Comput. Graphics 35}, 3 (2011), 677--684.
\newblock Shape Modeling International (SMI) Conference 2011.

\bibitem{ChenChenWangXuYongPaul2009jour}
{\sc Chen, X.-D., Chen, L., Wang, Y., Xu, G., Yong, J.-H., and Paul, J.-C.}
\newblock Computing the minimum distance between two {B}\'ezier curves.
\newblock {\em J. Comput. Appl. Math. 229}, 1 (2009), 294--301.

\bibitem{CheneyGoldstein1959jour}
{\sc Cheney, W., and Goldstein, A.~A.}
\newblock Proximity maps for convex sets.
\newblock {\em Proc. Amer. Math. Soc. 10\/} (1959), 448--450.

\bibitem{CSKM2013}
{\sc Chiu, S.~N., Stoyan, D., Kendall, W.~S., and Mecke, J.}
\newblock {\em Stochastic {G}eometry and its {A}pplications}, third~ed.
\newblock John Wiley \& Sons, Chichester, UK, 2013.

\bibitem{Clarkson1936jour}
{\sc Clarkson, J.~A.}
\newblock Uniformly convex spaces.
\newblock {\em Trans. Amer. Math. Soc. 40\/} (1936), 396--414.

\bibitem{Combettes1994jour}
{\sc Combettes, P.~L.}
\newblock Inconsistent signal feasibility problems: least-squares solutions in
  a product space.
\newblock {\em IEEE Trans. Signal Process 42\/} (1994), 2955--2966.

\bibitem{Combettes1996jour(CFP)}
{\sc Combettes, P.~L.}
\newblock {\em The Convex Feasibility Problem in Image Recovery}, vol.~95 of
  {\em Adv. Imaging Electron Phys.}
\newblock Elsevier, 1996, pp.~155--270.

\bibitem{CombettesBondon1999jour}
{\sc Combettes, P.~L., and Bondon, P.}
\newblock Hard-constrained inconsistent signal feasibility problems.
\newblock {\em IEEE Trans. Signal Process 47\/} (1999), 2460--2468.

\bibitem{Conway1990book}
{\sc Conway, J.~B.}
\newblock {\em A {C}ourse in {F}unctional {A}nalysis}, second~ed., vol.~96 of
  {\em Graduate Texts in Mathematics}.
\newblock Springer-Verlag, New York, 1990.
\newblock corrected fourth printing, 1997.

\bibitem{ConwaySloane}
{\sc Conway, J.~H., and Sloane, N. J.~A.}
\newblock {\em Sphere {P}ackings, {L}attices, and {G}roups}, third~ed.
\newblock Springer-Verlag, New York, 1999.

\bibitem{DamaBajracharya2018jour}
{\sc Damai, G.~R., and Bajracharya, P.~M.}
\newblock Mutually nearest points for two sets in metric spaces.
\newblock {\em Int. J. Theor. Appl. Math. 4}, 3 (2018), 29--34.

\bibitem{Dax2006jour}
{\sc Dax, A.}
\newblock The distance between two convex sets.
\newblock {\em Linear Algebra Appl. 416\/} (2006), 184--213.

\bibitem{De-BlasiMyjakPapini1992jour}
{\sc De~Blasi, F.~S., Myjak, J., and Papini, P.~L.}
\newblock On mutually nearest and mutually furthest points of sets in {B}anach
  spaces.
\newblock {\em J. Approx. Theory 70}, 2 (1992), 142--155.

\bibitem{Deutsch2001book}
{\sc Deutsch, F.}
\newblock {\em Best {A}pproximation in {I}nner {P}roduct {S}paces}, vol.~7 of
  {\em CMS Books in Mathematics/Ouvrages de Math\'{e}matiques de la SMC}.
\newblock Springer-Verlag, New York, 2001.

\bibitem{DigarKosuru2020jour}
{\sc Digar, A., and Raju~Kosuru, G.~S.}
\newblock Existence of best proximity pairs and a generalization of
  {C}arath\'{e}odory theorem.
\newblock {\em Numer. Funct. Anal. Optim. 41\/} (2020), 1901--1911.

\bibitem{DobkinKirkpatrick1985jour}
{\sc Dobkin, D.~P., and Kirkpatrick, D.~G.}
\newblock A linear algorithm for determining the separation of convex
  polyhedra.
\newblock {\em J. Algorithms 6}, 3 (1985), 381--392.

\bibitem{VoronoiCVD_Review}
{\sc Du, Q., Faber, V., and Gunzburger, M.}
\newblock Centroidal {V}oronoi tessellations: applications and algorithms.
\newblock {\em SIAM Rev. 41\/} (1999), 637--676.

\bibitem{DunfordSchwartz1958book}
{\sc Dunford, N., and Schwartz, J.~T.}
\newblock {\em Linear {O}perators. {I}. {G}eneral {T}heory}.
\newblock With the assistance of W. G. Bade and R. G. Bartle. Pure and Applied
  Mathematics, Vol. 7. Interscience Publishers, Inc., New York; London, 1958.

\bibitem{DuPreKass1992jour}
{\sc DuPr\'{e}, A.~M., and Kass, S.}
\newblock Distance and parallelism between flats in {${\bf R}^n$}.
\newblock {\em Linear Algebra Appl. 171\/} (1992), 99--107.

\bibitem{EhmannLin2000inproc}
{\sc Ehmann, S.~A., and Lin, M.~C.}
\newblock Accelerated proximity queries between convex polyhedra by multi-level
  {V}oronoi marching.
\newblock In {\em Proceedings of the 2000 IEEE/RSJ International Conference on
  Intelligent Robots and Systems (IROS 2000)\/} (2000), vol.~3, pp.~2101--2106.

\bibitem{ElberGrandine2008inproc}
{\sc Elber, G., and Grandine, T.}
\newblock Hausdorff and minimal distances between parametric freeforms in
  $\mathbb{R}^2$ and $\mathbb{R}^3$.
\newblock In {\em Advances in Geometric Modeling and Processing\/} (Berlin,
  Heidelberg, 2008), F.~Chen and B.~J{\"u}ttler, Eds., Springer Berlin
  Heidelberg, pp.~191--204.

\bibitem{FacasVicenteGoncalvesVitoria2014jour}
{\sc Facas~Vicente, M.~A., Gon\c{c}alves, A., and Vit\'{o}ria, J.}
\newblock Euclidean distance between two linear varieties.
\newblock {\em Appl. Math. Sci. (Ruse) 8}, 21-24 (2014), 1039--1043.

\bibitem{FanWangTongLiTang2024inproc}
{\sc Fan, P., Wang, W., Tong, R., Li, H., and Tang, M.}
\newblock g{D}ist: Efficient distance computation between 3{D} meshes on {GPU}.
\newblock In {\em SIGGRAPH Asia 2024 Conference Papers\/} (New York, NY, USA,
  2024), SA '24, Association for Computing Machinery.

\bibitem{Garkavi1970jour}
{\sc Garkavi, A.~L.}
\newblock The theory of best approximation in normed linear spaces.
\newblock 83--150.
\newblock Matemati\v ceski\u i\ Analiz 1967 (Russian), pp. 75-132; translated
  in Progress in Mathematics, Vol. 8: Mathematical Analysis (pp. 83--150).

\bibitem{GholamiTetruashviliStromCensor2013jour}
{\sc Gholami, M., Tetruashvili, L., Str\"om, E., and Censor, Y.}
\newblock Cooperative wireless sensor network positioning via implicit convex
  feasibility.
\newblock {\em IEEE Trans. Signal Process. 61\/} (2013), 5830--5840.

\bibitem{GilbertJohnsonKeerthi1988jour}
{\sc Gilbert, E.~G., Johnson, D.~W., and Keerthi, S.~S.}
\newblock A fast procedure for computing the distance between complex objects
  in three-dimensional space.
\newblock {\em IEEE J. Robot. Autom. 4}, 2 (1988), 193--203.

\bibitem{GoebelReich1984book}
{\sc Goebel, K., and Reich, S.}
\newblock {\em Uniform {C}onvexity, {H}yperbolic {G}eometry, and {N}onexpansive
  {M}appings}, vol.~83 of {\em Monographs and Textbooks in Pure and Applied
  Mathematics}.
\newblock Marcel Dekker Inc., New York, 1984.

\bibitem{GoldburgMarks1985jour}
{\sc Goldburg, M., and Marks, II, R.~J.}
\newblock Signal synthesis in the presence of an inconsistent set of
  constraints.
\newblock {\em IEEE Trans. Circuits and Systems 32\/} (1985), 647--663.

\bibitem{Goldstein1967book}
{\sc Goldstein, A.~A.}
\newblock {\em Constructive {R}eal {A}nalysis}.
\newblock Harper \& Row, Publishers, New York-London, 1967.

\bibitem{GoncalvesFacasVicenteVitoria2015jour}
{\sc Gon\c{c}alves, A., Facas~Vicente, M.~A., and Vit\'{o}ria, J.}
\newblock Optimal pair of two linear varieties.
\newblock {\em Appl. Math. Sci. (Ruse) 9\/} (2015), 593--596.

\bibitem{GrossTrenkler1996jour}
{\sc Gross, J., and Trenkler, G.}
\newblock On the least squares distance between affine subspaces.
\newblock {\em Linear Algebra Appl. 237/238\/} (1996), 269--276.
\newblock Special issue honoring Calyampudi Radhakrishna Rao.

\bibitem{GruberLek}
{\sc Gruber, P.~M., and Lekkerkerker, C.~G.}
\newblock {\em Geometry of {N}umbers}, second~ed.
\newblock North Holland, 1987.

\bibitem{GubinPolyakRaik1967jour}
{\sc Gubin, L., Polyak, B., and Raik, E.}
\newblock The method of projections for finding the common point of convex
  sets.
\newblock {\em USSR Comput. Math. Math. Phys. 7\/} (1967), 1--24.

\bibitem{Holmes1975book}
{\sc Holmes, R.~B.}
\newblock {\em Geometric {F}unctional {A}nalysis and its {A}pplications},
  vol.~No. 24 of {\em Graduate Texts in Mathematics}.
\newblock Springer-Verlag, New York-Heidelberg, 1975.

\bibitem{James1964jour}
{\sc James, R.~C.}
\newblock Characterizations of reflexivity.
\newblock {\em Studia Math. 23\/} (1963/64), 205--216.

\bibitem{James1972jour}
{\sc James, R.~C.}
\newblock Reflexivity and the sup of linear functionals.
\newblock {\em Israel J. Math. 13\/} (1972), 289--300.

\bibitem{Jameson1974book}
{\sc Jameson, G. J.~O.}
\newblock {\em Topology and {N}ormed {S}paces}.
\newblock Chapman and Hall, London; Halsted Press [John Wiley \& Sons, Inc.],
  New York, 1974.

\bibitem{JohnsonCohen1998inproc}
{\sc Johnson, D., and Cohen, E.}
\newblock A framework for efficient minimum distance computations.
\newblock In {\em Proceedings. 1998 IEEE International Conference on Robotics
  and Automation\/} (1998), vol.~4, pp.~3678--3684.

\bibitem{KirkReichVeeramani2003jour}
{\sc Kirk, W.~A., Reich, S., and Veeramani, P.}
\newblock Proximinal retracts and best proximity pair theorems.
\newblock {\em Numer. Funct. Anal. Optim. 24\/} (2003), 851--862.

\bibitem{KopeckaReich2004jour}
{\sc Kopeck\'{a}, E., and Reich, S.}
\newblock A note on the von {N}eumann alternating projections algorithm.
\newblock {\em J. Nonlinear Convex Anal. 5\/} (2004), 379--386.

\bibitem{KopeckaReich2012jour}
{\sc Kopeck\'{a}, E., and Reich, S.}
\newblock A note on alternating projections in {H}ilbert space.
\newblock {\em J. Fixed Point Theory Appl. 12\/} (2012), 41--47.

\bibitem{Kothe1969book}
{\sc K{\"o}the, G.}
\newblock {\em Topological {V}ector {S}paces. {I}}.
\newblock Translated from the German edition by D. J. H. Garling. Die
  Grundlehren der mathematischen Wissenschaften, Band 159. Springer-Verlag, New
  York, USA, 1969.

\bibitem{Kremp1986jour}
{\sc Kremp, S.}
\newblock An elementary proof of the {E}berlein-{{\v S}}mulian theorem and the
  double limit criterion.
\newblock {\em Arch. Math. (Basel) 47}, 1 (1986), 66--69.

\bibitem{LeviShafrir2014jour}
{\sc Levi, S., and Shafrir, I.}
\newblock On the distance between homotopy classes of maps between spheres.
\newblock {\em J. Fixed Point Theory Appl. 15}, 2 (2014), 501--518.

\bibitem{LevitinPolyak1966jour}
{\sc Levitin, E.~S., and Polyak, B.~T.}
\newblock Constrained minimization methods.
\newblock {\em USSR Comput. Math. Math. Phys. 6}, 5 (1966), 1--50.
\newblock Published before in Russian in Zh. vychisl. Mat. mat. Fiz 6 (1966),
  787--823.

\bibitem{Li2000jour}
{\sc Li, C.}
\newblock On mutually nearest and mutually furthest points in reflexive
  {B}anach spaces.
\newblock {\em J. Approx. Theory 103}, 1 (2000), 1--17.

\bibitem{LiXu2003jour}
{\sc Li, C., and Xu, H.-K.}
\newblock Porosity of mutually nearest and mutually furthest points in {B}anach
  spaces.
\newblock {\em J. Approx. Theory 125}, 1 (2003), 10--25.

\bibitem{Lin1966jour}
{\sc Lin, B.-l.}
\newblock Distance-sets in normed vector spaces.
\newblock {\em Nieuw Arch. Wisk. (3) 14\/} (1966), 23--30.

\bibitem{LinGottschalk1998inproc}
{\sc Lin, M.~C., and Gottschalk, S.}
\newblock Collision detection between geometric models: A survey.
\newblock In {\em Proc. of IMA Conference on Mathematics of Surfaces\/} (1998),
  vol.~1, pp.~602--608.

\bibitem{LinManochaKim2018inbook}
{\sc Lin, M.~C., Manocha, D., and Kim, Y.~J.}
\newblock Collision and proximity queries.
\newblock In {\em Handbook of Discrete and Computational Geometry}, J.~E.
  Goodman, J.~O'Rourke, and C.~D. T\'oth, Eds. CRC Press, Boca Raton, FL, 2018,
  ch.~39.
\newblock (28 pages).

\bibitem{LindenstraussTzafriri1979book}
{\sc Lindenstrauss, J., and Tzafriri, L.}
\newblock {\em Classical {B}anach Spaces, {II}: {F}unction Spaces}.
\newblock Ergebnisse der Mathematik und ihrer Grenzgebiete [Results in
  Mathematics and Related Areas]. Springer-Verlag, Berlin-New York, 1979.

\bibitem{Luke2008jour}
{\sc Luke, D.~R.}
\newblock Finding best approximation pairs relative to a convex and
  prox-regular set in a {H}ilbert space.
\newblock {\em SIAM J. Optim. 19\/} (2008), 714--739.

\bibitem{Luo2014jour}
{\sc Luo, X.}
\newblock Characterizations and uniqueness of mutually nearest points for two
  sets in normed spaces.
\newblock {\em Numer. Funct. Anal. Optim. 35}, 5 (2014), 611--622.

\bibitem{LuoLiXuYao2011jour}
{\sc Luo, X.-F., Li, C., Xu, H.-K., and Yao, J.-C.}
\newblock Existence of best simultaneous approximations in {$L_p(S,\Sigma,X)$}.
\newblock {\em J. Approx. Theory 163}, 9 (2011), 1300--1316.

\bibitem{MatouskovaReich2003jour}
{\sc Matou{\v s}kov\'a, E., and Reich, S.}
\newblock The {H}undal example revisited.
\newblock {\em J. Nonlinear Convex Anal. 4}, 3 (2003), 411--427.

\bibitem{MoralesSilvaGao2017inbook}
{\sc Morales-Silva, D., and Gao, D.~Y.}
\newblock On minimal distance between two surfaces.
\newblock In {\em Canonical Duality Theory: Unified Methodology for
  Multidisciplinary Study}, D.~Y. Gao, V.~Latorre, and N.~Ruan, Eds. Springer
  International Publishing, Cham, 2017, pp.~359--371.

\bibitem{Myers1945jour}
{\sc Myers, S.~B.}
\newblock Arcs and geodesics in metric spaces.
\newblock {\em Trans. Amer. Math. Soc. 57\/} (1945), 217--227.

\bibitem{Narang1976jour}
{\sc Narang, T.~D.}
\newblock On distance sets.
\newblock {\em Indian J. Pure Appl. Math. 7}, 10 (1976), 1137--1141.

\bibitem{Narang1983jour}
{\sc Narang, T.~D.}
\newblock On distance and distant sets.
\newblock {\em Math. Ed. (Siwan) 17}, 3 (1983), 87--88.

\bibitem{Narang1984jour}
{\sc Narang, T.~D.}
\newblock On proximal pairs.
\newblock {\em Indian J. Pure Appl. Math. 15}, 3 (1984), 251--254.

\bibitem{Narang1991jour}
{\sc Narang, T.~D.}
\newblock Some remarks on a paper of {K}. {B}. {X}u: ``{A} result on best
  proximity pair of two sets'' [{J}. {A}pprox. {T}heory {\bf 54} (1988), no. 3,
  322--325; {MR}0960054 (90a:41034)].
\newblock {\em Math. Ed. (Siwan) 25\/} (1991), 126--128.

\bibitem{Nicolescu1938jour}
{\sc Nicolescu, M.}
\newblock Sur la meilleure approximation d'une fonction donnee par les
  functions d'une famille donnee.
\newblock {\em Bull. Fac. Sti. Cernauti 12\/} (1938), 120--128.

\bibitem{OkabeBootsSugiharaChiu2000book}
{\sc Okabe, A., Boots, B., Sugihara, K., and Chiu, S.~N.}
\newblock {\em Spatial {T}essellations: {C}oncepts and {A}pplications of
  {V}oronoi {D}iagrams}, second~ed.
\newblock Wiley Series in Probability and Statistics. John Wiley \& Sons Ltd.,
  Chichester, 2000.
\newblock With a foreword by D. G. Kendall.

\bibitem{Pai1974jour}
{\sc Pai, D.~V.}
\newblock Proximal points of convex sets in normed linear spaces.
\newblock {\em Yokohama Math. J. 22\/} (1974), 53--78.

\bibitem{PatelPatel2023jour}
{\sc Patel, D.~K., and Patel, B.}
\newblock Finding the best proximity point of generalized multivalued
  contractions with applications.
\newblock {\em Numer. Funct. Anal. Optim. 44\/} (2023), 1602--1627.

\bibitem{PatogluGillespie2002inproc}
{\sc Patoglu, V., and Gillespie, R.}
\newblock Extremal distance maintenance for parametric curves and surfaces.
\newblock In {\em Proceedings of the 2002 IEEE International Conference on
  Robotics and Automation\/} (2002), vol.~3, pp.~2817--2823.

\bibitem{Peng2012jour}
{\sc Peng, L.~H.}
\newblock Porosity of mutually nearest and mutually furthest points problems in
  geodesic space.
\newblock {\em Numer. Funct. Anal. Optim. 33}, 3 (2012), 284--300.

\bibitem{Phelps1960jour}
{\sc Phelps, R.~R.}
\newblock Uniqueness of {H}ahn-{B}anach extensions and unique best
  approximation.
\newblock {\em Trans. Amer. Math. Soc. 95\/} (1960), 238--255.

\bibitem{Prus2001incol}
{\sc Prus, S.}
\newblock Geometrical background of metric fixed point theory.
\newblock In {\em Handbook of {M}etric {F}ixed Point Theory}, W.~A. Kirk and
  B.~Sims, Eds. Kluwer Acad. Publ., Dordrecht, 2001, pp.~93--132.

\bibitem{Quinlan1994inproc}
{\sc Quinlan, S.}
\newblock Efficient distance computation between non-convex objects.
\newblock In {\em Proceedings of the 1994 IEEE International Conference on
  Robotics and Automation\/} (1994), vol.~4, pp.~3324--3329.

\bibitem{RajuKosuruVeeramani2010jour}
{\sc Raju~Kosuru, G.~S., and Veeramani, P.}
\newblock On existence of best proximity pair theorems for relatively
  nonexpansive mappings.
\newblock {\em J. Nonlinear Convex Anal. 11\/} (2010), 71--77.

\bibitem{ReemISVD2009proc}
{\sc Reem, D.}
\newblock An algorithm for computing {V}oronoi diagrams of general generators
  in general normed spaces.
\newblock In {\em Proceedings of the sixth annual {IEEE} International
  Symposium on {V}oronoi Diagrams in Science and Engineering ({ISVD} 2009),
  Copenhagen, Denmark\/} (2009), pp.~144--152.

\bibitem{Reem2018jour}
{\sc Reem, D.}
\newblock On the computation of zone and double zone diagrams.
\newblock {\em Discrete Comput. Geom. 59\/} (2018), 253--292.
\newblock arXiv:1208.3124 [cs.CG] (2012) (current version: [v6], 31 Dec 2017).

\bibitem{Reem2023jour}
{\sc Reem, D.}
\newblock The projector algorithm: a simple parallel algorithm for computing
  {V}oronoi diagrams and {D}elaunay graphs.
\newblock {\em Theoret. Comput. Sci. 970\/} (2023), Paper No. 114054 (38
  pages).

\bibitem{Rockafellar1970book}
{\sc Rockafellar, R.~T.}
\newblock {\em Convex {A}nalysis}.
\newblock Princeton Mathematical Series, No. 28. Princeton University Press,
  Princeton, NJ, USA, 1970.

\bibitem{RubinsteinShafrir2007jour}
{\sc Rubinstein, J., and Shafrir, I.}
\newblock The distance between homotopy classes of {$S^1$}-valued maps in
  multiply connected domains.
\newblock {\em Israel J. Math. 160\/} (2007), 41--59.

\bibitem{SadiqBashaVeeramani200jour}
{\sc Sadiq~Basha, S., and Veeramani, P.}
\newblock Best proximity pair theorems for multifunctions with open fibres.
\newblock {\em J. Approx. Theory 103\/} (2000), 119--129.

\bibitem{SadiqBashaVeeramaniPai2001jour}
{\sc Sadiq~Basha, S., Veeramani, P., and Pai, D.~V.}
\newblock Best proximity pair theorems.
\newblock {\em Indian J. Pure Appl. Math. 32\/} (2001), 1237--1246.

\bibitem{SahneySingh1980jour}
{\sc Sahney, B.~N., and Singh, S.~P.}
\newblock On best simultaneous approximation.
\newblock In {\em Approximation theory, {III} ({P}roc. {C}onf., {U}niv.
  {T}exas, {A}ustin, {T}ex., 1980)}, E.~W. Cheney, Ed. Academic Press, New
  York-London, 1980, pp.~783--789.

\bibitem{Sankar-RajAnthony-Eldred2014jour}
{\sc Sankar~Raj, V., and Anthony~Eldred, A.}
\newblock A characterization of strictly convex spaces and applications.
\newblock {\em J. Optim. Theory Appl. 160}, 2 (2014), 703--710.

\bibitem{SatoHirataMaruyamaArita1996inproc}
{\sc Sato, Y., Hirata, M., Maruyama, T., and Arita, Y.}
\newblock Efficient collision detection using fast distance-calculation
  algorithms for convex and non-convex objects.
\newblock In {\em Proceedings of IEEE International Conference on Robotics and
  Automation\/} (1996), vol.~1, pp.~771--778.

\bibitem{Schwartz1981jour}
{\sc Schwartz, J.~T.}
\newblock Finding the minimum distance between two convex polygons.
\newblock {\em Inform. Process. Lett. 13}, 4-5 (1981), 168--170.

\bibitem{SharmaChandok2026jour}
{\sc Sharma, S., and Chandok, S.}
\newblock Existence of best proximity pair for class of enriched type
  nonexpansive mappings with applications.
\newblock {\em Numer. Funct. Anal. Optim. 46}, 10-12 (2025), 796--815.

\bibitem{Singer1970book}
{\sc Singer, I.}
\newblock {\em Best {A}pproximation in {N}ormed {L}inear {S}paces by {E}lements
  of {L}inear {S}ubspaces}, vol.~Band 171 of {\em Die Grundlehren der
  mathematischen Wissenschaften}.
\newblock Publishing House of the Academy of the Socialist Republic of Romania,
  Bucharest; Springer-Verlag, New York-Berlin, 1970.
\newblock Translated from the Romanian by Radu Georgescu.

\bibitem{SonYoonKimElber2020jour}
{\sc Son, S.-H., Yoon, S.-H., Kim, M.-S., and Elber, G.}
\newblock Efficient minimum distance computation for solids of revolution.
\newblock {\em Comput. Graphics Forum 39}, 2 (2020), 535--544.
\newblock (Special issue dedicated to Eurographics 2020).

\bibitem{Stiles1965b-jour}
{\sc Stiles, W.~J.}
\newblock Closest-point maps and their products. {II}.
\newblock {\em Nieuw Arch. Wisk. (3) 13\/} (1965), 212--225.

\bibitem{SuzukiKikkawaVetro2009jour}
{\sc Suzuki, T., Kikkawa, M., and Vetro, C.}
\newblock The existence of best proximity points in metric spaces with the
  property {UC}.
\newblock {\em Nonlinear Anal. 71\/} (2009), 2918--2926.

\bibitem{VanTiel1984book}
{\sc van Tiel, J.}
\newblock {\em {C}onvex {A}nalysis: {A}n {I}ntroductory {T}ext}.
\newblock John Wiley and Sons, Universities Press, Belfast, Northern Ireland,
  1984.

\bibitem{Vlasov1973jour}
{\sc Vlasov, L.~P.}
\newblock Approximative properties of sets in normed linear spaces.
\newblock {\em Uspehi Mat. Nauk 28}, 6 (174) (1973), 3--66.

\bibitem{VoiseiZalinescu2011jour}
{\sc Voisei, M.~D., and Z{\u a}linescu, C.}
\newblock A counter-example to `minimal distance between two non-convex
  surfaces'.
\newblock {\em Optimization 60}, 5 (2011), 593--602.

\bibitem{vonNeumann1949jour}
{\sc von Neumann, J.}
\newblock On rings of operators. {R}eduction theory.
\newblock {\em Ann. of Math. 50}, 2 (1949), 401--485.
\newblock (Most of the paper was ready in 1938, Lemma 22 (alternating
  projections) was used in 1935).

\bibitem{vonNeumann1950book}
{\sc von Neumann, J.}
\newblock {\em Functional {O}perators. {II}. {T}he {G}eometry of {O}rthogonal
  {S}paces}, vol.~No. 22 of {\em Annals of Mathematics Studies}.
\newblock Princeton University Press, Princeton, NJ, 1950.
\newblock Based on lecture notes from 1933.

\bibitem{Whitley1967jour}
{\sc Whitley, R.}
\newblock An elementary proof of the {E}berlein-{{\v S}}mulian theorem.
\newblock {\em Math. Ann. 172\/} (1967), 116--118.

\bibitem{Willner1968jour}
{\sc Willner, L.~B.}
\newblock On the distance between polytopes.
\newblock {\em Quart. Appl. Math. 26\/} (1968), 207--212.

\bibitem{Xu1983jour}
{\sc Xu, X.}
\newblock On best proximity pairs and best mutual approximations ({C}hinese).
\newblock {\em J. Zhejiang Normal Univ. 6\/} (1983), pp. 23 and 49--54.

\bibitem{Xu1988jour}
{\sc Xu, X.}
\newblock A result on best proximity pair of two sets.
\newblock {\em J. Approx. Theory 54\/} (1988), 322--325.

\bibitem{YoulaVelasco1986jour}
{\sc Youla, D.~C., and Velasco, V.}
\newblock Extensions of a result on the synthesis of signals in the presence of
  inconsistent constraints.
\newblock {\em IEEE Trans. Circuits and Systems 33\/} (1986), 465--468.

\bibitem{Zalinescu2022jour}
{\sc Z{\u a}linescu, C.}
\newblock On canonical duality theory and constrained optimization problems.
\newblock {\em J. Global Optim. 82}, 4 (2022), 1053--1070.

\bibitem{Zaslavski2022jour}
{\sc Zaslavski, A.~J.}
\newblock The method of cyclic projections for closed convex sets in a
  {H}ilbert space under the presence of computational errors.
\newblock {\em Numer. Algorithms 91}, 3 (2022), 1427--1439.

\bibitem{Zaslavski2024book}
{\sc Zaslavski, A.~J.}
\newblock {\em Solutions of {F}ixed {P}oint {P}roblems with {C}omputational
  {E}rrors}, vol.~210 of {\em Springer Optimization and Its Applications}.
\newblock Springer, Cham, 2024.

\bibitem{ZeghloulRambeaud1996jour}
{\sc Zeghloul, S., and Rambeaud, P.}
\newblock A fast algorithm for distance calculation between convex objects
  using the optimization approach.
\newblock {\em Robotica 14}, 4 (1996), 355--363.

\end{thebibliography}


\end{document}